\documentclass[reqno,12pt]{amsart}
\theoremstyle{plain}
\newtheorem{thm}{Theorem}[section]
\newtheorem{thmsub}{Theorem}[subsection]
\newtheorem{cor}[thm]{Corollary}
\newtheorem{corsub}[thmsub]{Corollary}
\newtheorem{lemma}[thm]{Lemma} 
\newtheorem{lemmasub}[thmsub]{Lemma}
\newtheorem{prop}[thm]{Proposition}
\newtheorem{propsub}[thmsub]{Proposition}

\theoremstyle{remark}

\newtheorem{remark}[thm]{Remark}
\newtheorem{remarksub}[thmsub]{Remark}
\newtheorem{remarks}[thm]{Remarks}

\theoremstyle{definition}
\newtheorem{defi}[thm]{Definition}
\newtheorem{defisub}[thmsub]{Definition}
\newtheorem{example}[thm]{Example}
\newtheorem{examplesub}[thmsub]{Example}

\newtheorem{obssub}[thmsub]{Observation}

\newtheorem{ques}[thm]{Question}
\newtheorem{quessub}[thmsub]{Question}

\usepackage{amscd,amssymb,comment,epic,eepic,euscript,graphics}
\usepackage{enumerate}
\usepackage[initials]{amsrefs}
\usepackage{stmaryrd}
\usepackage[all]{xy}
\usepackage{color}

\newcount\theTime
\newcount\theHour
\newcount\theMinute
\newcount\theMinuteTens
\newcount\theScratch
\theTime=\number\time
\theHour=\theTime
\divide\theHour by 60
\theScratch=\theHour
\multiply\theScratch by 60
\theMinute=\theTime
\advance\theMinute by -\theScratch
\theMinuteTens=\theMinute
\divide\theMinuteTens by 10
\theScratch=\theMinuteTens
\multiply\theScratch by 10
\advance\theMinute by -\theScratch
\def\timeHHMM{{\number\theHour:\number\theMinuteTens\number\theMinute}}

\def\today{{\number\day\space
 \ifcase\month\or
  January\or February\or March\or April\or May\or June\or
  July\or August\or September\or October\or November\or December\fi
 \space\number\year}}

\def\timeanddate{{\timeHHMM\space o'clock, \today}}
\newcommand\Ac{{\mathcal{A}}}
\newcommand\Ad{{\rm Ad}}
\newcommand\Afr{{\mathfrak A}}
\newcommand\ah{{\hat a}}

\newcommand\Aut{{\rm Aut}}
\newcommand\At{{\widetilde A}}
\newcommand\Bc{{\mathcal{B}}}
\newcommand\betah{{\hat\beta}}
\newcommand\Bfr{{\mathfrak B}}
\newcommand\bh{{\hat b}}
\newcommand\Cc{{\mathcal{C}}}

\newcommand\conv{\operatorname{conv}}
\newcommand\Cpx{{\mathbf C}}
\newcommand\Dc{{\mathcal{D}}}

\newcommand\Dt{{\widetilde D}}
\newcommand\eh{{\hat e}}
\newcommand\End{{\rm End}}
\newcommand\eps{\epsilon}

\newcommand\ev{{\operatorname{ev}}}

\newcommand\Fc{{\mathcal{F}}}

\newcommand\freeprod{\operatornamewithlimits{\ast}}

\newcommand\HEu{{\EuScript H}}                   
\newcommand\id{{\operatorname{id}}}
\newcommand\KEu{{\EuScript K}}                   
\newcommand\KEut{{\widetilde{\KEu}}}            
\newcommand\Kt{{\widetilde K}}
\newcommand\lambdat{{\tilde\lambda}}
\newcommand\lspan{\mathrm{span}\,}
\newcommand\Mcal{{\mathcal{M}}}
\newcommand\Nats{{\mathbf N}}
\newcommand\Nc{{\mathcal{N}}}
\newcommand\NC{\operatorname{NC}}
\newcommand\oneh{{\hat 1}}
\newcommand\oup{^{\mathrm o}}
\newcommand\Pc{{\mathcal{P}}}

\newcommand\phichk{{\check\phi}}

\newcommand\phit{{\tilde\phi}}
\newcommand\pit{{\tilde\pi}}

\newcommand\psih{{\hat\psi}}

\newcommand\qchk{{\check q}}
\newcommand\QEu{{\EuScript Q}}
\newcommand\QSS{{\operatorname{QSS}}}
\newcommand\restrict{{\upharpoonright}}
\newcommand\REu{{\EuScript R}}                   
\newcommand\rhoh{{\hat\rho}}

\newcommand\SSt{{\operatorname{SS}}}
\newcommand\stalg{{*\operatorname{-alg}}}
\newcommand\tauchk{{\check\tau}}
\newcommand\Tc{{\mathcal{T}}}
\newcommand\TQSS{{\operatorname{TQSS}}}
\newcommand\tr{{\mathrm{tr}}}
\newcommand\TVEu{{\EuScript{T\hspace{0.05em}V}}}
\newcommand\TWEu{{\EuScript{T\hspace{0.05em}W}}}
\newcommand\ut{{\tilde u}}
\newcommand\Vc{{\mathcal{V}}}
\newcommand\VEu{{\EuScript V}}                   
\newcommand\vh{{\hat v}}
\newcommand\Wc{{\mathcal{W}}}
\newcommand\wh{{\hat w}}
\newcommand\xh{{\hat x}}
\newcommand\yh{{\hat y}}
\newcommand\zh{{\hat z}}
\newcommand\ZQSS{{\operatorname{ZQSS}}}
\newcommand\ZTQSS{{\operatorname{ZTQSS}}}

\begin{document}

\title[Quantum symmetric states, \timeanddate]{Quantum symmetric states on free product $C^*$-algebras}

\author[Dykema]{Kenneth J.\ Dykema$^{*}$}
\address{K.\ Dykema, Department of Mathematics, Texas A\&M University,
College Station, TX 77843-3368, USA}
\email{kdykema@math.tamu.edu}
\thanks{\footnotesize $^{*}$Research supported in part by NSF grant DMS-1202660.
$^\dag$Research supported in part by EPSRC grant EP/H016708/1}

\author[K\"ostler]{Claus K\"ostler$^\dag$}
\address{C.\ K\"ostler,
School of Mathematical Sciences, Western Gateway Building, Western Road, University College Cork,
Cork, Ireland}
\email{claus@ucc.ie}

\author[Williams, \timeanddate]{John D.\ Williams}
\address{J.\ Williams, Department of Mathematics,
Fachrichtung Mathematik, Universit\"at des Saarlandes, 
Campus E24, 66123 Saarbr\"ucken, Germany}
\email{williams@math.uni-sb.de}


\date{22 September 2014}

\begin{abstract}
We introduce symmetric states and quantum symmetric states on universal unital free product $C^*$-algebras of the form
$\Afr=\freeprod_1^\infty A$
for an arbitrary unital $C^*$-algebra $A$, as a generalization of the notions of exchangeable and quantum exchangeable random variables.
We prove existence of conditional expectations onto tail algebras in various settings and we define a natural $C^*$-subalgebra of the tail
algebra, called the tail $C^*$-algebra.
Extending and building on the proof of the noncommutative de Finetti theorem of K\"ostler and Speicher,
we prove a de Finetti type theorem that
characterizes quantum symmetric states in terms of amalgamated free products over the tail $C^*$-algebra,
and we provide a convenient description of the set of all quantum symmetric states on $\Afr$ in terms of 
$C^*$-algebras generated by homomorphic images of $A$ and the tail $C^*$-algebra.
This description allows a characterization of the extreme quantum symmetric states.
Similar results are proved for the subset of tracial quantum symmetric states, though in terms of von Neumann algebras
and normal conditional expectations.
The central quantum symmetric states are those for which the tail algebra is in the center of the von Neumann algebra, and we
show that the central quantum symmetric states form a Choquet simplex whose extreme points are the free product states,
while the tracial central quantum symmetric states form a Choquet simplex whose extreme points are the free product traces.
\end{abstract}

\maketitle

\section{Introduction}

Classical random variables $x_1,x_2,\ldots$ are said to be exchangeable
if their joint distribution is unchanged by permuting the variables.
The classical de Finetti theorem characterizes exchangeable random variables as those that are identically distributed
and conditionally independent over their tail $\sigma$-algebra (see, e.g.,~\cite{Ka05}) and
Hewitt and Savage~\cite{HS55} rephrased this theorem in terms of symmetric measures on the infinite Cartesian product
of a compact Hausdorff space.
St\o{}rmer~\cite{St69} extended the purview and the result to the realm of $C^*$-algebras;
he considered the minimal tensor product $B=\bigotimes_1^\infty A$ of a unital $C^*$-algebra $A$ with itself
infinitely many times and defined a state on $B$ to be symmetric if it is invariant under all automorphisms of $B$ that result
from permuting the tensor factors (it is clearly equivalent for this purpose to speak of either finite permutations or
arbitrary permutations of the natural numbers).
He showed that symmetric states correspond to mixtures of infinite tensor product states; $\otimes_1^\infty\phi$,
where $\phi$ is a state on $A$.
In other words, these infinite tensor product states are the extreme points of the compact convex set of symmetric states on $B$,
and, therefore, every symetric state on $B$ is an integral with respect to a probability measure of infinite tensor product states.
Furthermore, St\o{}rmer showed that the symmetric states on $B$ form a Choquet simplex, so this probability measure is unique.

The noncommutative de Finetti theorem was discovered by K\"ostler and Speicher in~\cite{KSp09}.
It concerns random variables $x_1,x_2,\ldots$ in a $W^*$-noncommutative probability space $(\Mcal,\phi)$,
which consists of a possibly noncommutative unital $W^*$-algebra $\Mcal$ and a normal state $\phi$ on $\Mcal$,
which in~\cite{KSp09} was assumed to be faithful;
the random variables are elements $x_j\in\Mcal$, and are often called noncommutative random variables
to emphasize that they need not commute with each other.
These noncommutative random variables are said to be exchangeable (or classically exchangeable) if
the joint distribution of them is invariant under permutations, namely, if
\[
\phi(x_{i_1}x_{i_2}\cdots x_{i_n})=\phi(x_{i_{\sigma(1)}}x_{i_{\sigma(2)}}\cdots x_{i_{\sigma(n)}})
\]
for all $n\in\Nats$, $i_1,\ldots,i_n\in\Nats$ and all permutations $\sigma$ of $\Nats$.
They are said to be quantum exchangeable if they are invariant under quantum permutations
(see~\cite{KSp09} or \S\ref{sec:prelims} below), namely,
under the natural co-action of the quantum permutation group of S.\ Wang~\cite{W98}.
K\"ostler and Speicher showed that noncommutative
random variables are quantum exchangeable if and only if they are free with amalgamation
over the tail algebra of the sequence.
In~\cite{DK}, the first two authors showed that every countably generated von Neumann algebra arises as the tail algebra of
a quantum exchangeable sequence of random variables.
In~\cite{K10} and~\cite{GK09}, exchangeable sequences and the more general notions of spreadability and braidability, respectively,
are investigated in the context of a von Neumann algebra equipped with a normal faithful state.

In this paper, we investigate exchangeability and quantum exchangeability for sequences of representations of a unital $C^*$-algebra $A$.
This naturally results in the notion of symmetric states and of quantum symmetric states on the universal unital $C^*$-algebra free product
$\Afr=\freeprod_{i=1}^\infty A$ of $A$ with itself infinitely many times (see \S\ref{sec:prelims} for definitions).
For a symmetric state $\psi$, we consider the natural notion of a tail algebra (see \S\ref{sec:tails}), which is a certain von Neumann subalgebra $\Tc_\psi$
of the von Neumann algebra $\Mcal_\psi$
generated by the image of $\Afr$ under Gelfand--Naimark--Segal (or GNS) representation of $\psi$.
We prove (\S\ref{subsec:tail}) the existence of a conditional expectation from a certain weakly dense $C^*$-subalgebra of $\Mcal_\psi$ onto $\Tc_\psi$
and we investigate when this extends to a normal conditional expectation from $\Mcal_\psi$ onto $\Tc_\psi$.
We show (\S\ref{subsec:normal})
that we can always find a normal conditional expectation whose domain is a slightly larger von Neumann algebra $\Nc_\psi$.
We also define (\S\ref{subsec:tailC*})
the tail $C^*$-algebra, which is the smallest $C^*$-subalgebra of $\Tc_\psi$ for which the conditional expectation can exist.
We also partially characterize (\S\ref{subsec:ex}) the extreme symmetric states.

We prove
an analogue of the noncommutative de Finetti theorem (see \S\ref{sec:freenessYields} and \S\ref{sec:freeness}),
namely, that the images of the copies of $A$ in the Gelfand--Naimark--Segal construction of the quantum symmetric state
are free with amalgamation over the tail algebra, in both $C^*$- and $W^*$-versions.
These are all natural extensions of the notions discussed above and many of our proofs follow closely those of~\cite{KSp09}.
K\"ostler and Speicher's noncommutative de Finetti theorem can be seen as a special case of ours,
by taking $A$ to be a suitable abelian $C^*$-algebra and requiring faithfulness of the state.

A similar extension of the noncommutative de Finetti theorem was obtained by Curran~\cite{C09}.
He considered quantum exchangeability of $*$-representations of a $*$-algebra and proved that, for an infinite sequence of such,
freeness with amalgamation over a tail algebra holds.
He actually considered the more delicate situation of quantum exchangeability of finite sequences of $*$-representations and obtained
the result for infinite sequences as a limiting case.
However, he did assume faithfulness of the state in the $W^*$-noncommutative probability space.
Although our results up to this point overlap with and are similar in spirit to Curran's, we have included our proof
because (a) it is different from that contained in~\cite{C09} (and closer to the proof of~\cite{KSp09}) and
(b) we must allow the states on the relevant $W^*$-algebras to be non-faithful.

To summarize, our results on quantum symmetric states as described up to this point are an extension to the realm of unital $C^*$-algebras of previous results
about quantum exchangeable random variables,
analogous to what St\o{}rmer did in~\cite{St69} for the classical de Finetti theorem.
An advantage of this extension
is that it allows for treatment of the situation of non-faithful states in the $C^*$-noncommutative
probability spaces (thus, overcoming a limitation of~\cite{KSp09}, \cite{C09} and~\cite{DK}).

Investigating further, we obtain a nice classification of quantum symmetric states on $\Afr$;
they are in bijection with the quintuples $(B,D,F,\theta,\rho)$, where $D\subseteq B$ is a unital inclusion of $C^*$-algebras,
$F:B\to D$ is a conditional expectation with faithful GNS representation, 
$\theta:A\to B$ is a unital $*$-homomorphism so that certain generation and minimality
properties are satisfied, and $\rho$ is a state on $D$ (see \S\ref{sec:descr}).
In this correspondence, $D$ plays the role of the tail $C^*$-algebra of the quantum symmetric state.
For tracial quantum symmetric states, we also have a similar characterization using von Neumann algebras.

One of our main goals in this paper
is the characterization of extreme points of the compact convex sets of quantum symmetric states
and of tracial quantum symmetric states, which is accomplished in Theorems~\ref{thm:extrQSS} and~\ref{thm:extrTQSS}
using the aforementioned classification found in~\S\ref{sec:descr}.

Finally, in~\S\ref{sec:ZQSS}, we consider the quantum symmetric states whose tail algebras are in the center of
the von Neumann algebra generated by the image of the GNS representation.
We call these the central quantum symmetric states, and show that they form a Choquet simplex, whose extreme points
are the free product states.
We show that also the central quantum symmetric states that are tracial form a Choquet simplex, and its extreme
points are the free product tracial states.

\medskip
\noindent
{\bf Acknowledgement.}
The authors thank Weihua Liu for pointing out a mistake in an earlier version of this paper (regarding existence of normal conditional expectations)
and for showing us Example~\ref{ex:Liu}.

\section{Preliminaries and Definitions}
\label{sec:prelims}

The quantum permutation group (on $k$ letters) was found by S.\ Wang~\cite{W98}.
It is the universal unital $C^*$-algebra $A_s(k)$
generated projections $u_{ij}$ (for $1\le i,j\le k$)  satisfying $\sum_{i=1}^ku_{ij}=1$ for all $j$ and
$\sum_{j=1}^ku_{ij}=1$ for all $i$.
The abelian $C^*$-algebra $C(S_k)$ of continuous functions on the symmetric group $S_k$ on $k$ letters
is a quotient of $A_s(k)$ under the map (the standard ``abelianization'' representation)
that sends $u_{ij}$ to the characteristic function of the set of permutations
that send $j$ to $i$.

\begin{defi}\label{def:exch}
Let $A$ be a unital
$C^*$-algebra and 
let $(\Bfr,\phi)$ be a $C^*$-noncommutative probability space, namely, a unital $C^*$-algebra $\Bfr$ equipped
with a state $\phi$.
For every $i\in\Nats$ let $\pi_i:A\to\Bfr$ be a unital $*$-homomorphism.
We say $(\pi_i)_{i=1}^\infty$ is {\em exchangeable} (or classically exchangeable)
with respect to $\phi$ if for all permutations
$\sigma$ of $\Nats$,
all $n\in\Nats$, all $i(1),\ldots,i(n)\in\Nats$ and all $a_1,\ldots,a_n\in A$,
we have
\begin{equation}\label{eq:exch}
\phi(\pi_{i(1)}(a_1)\cdots\pi_{i(n)}(a_n))=\phi(\pi_{\sigma(i(1))}(a_1)\cdots\pi_{\sigma(i(n))}(a_n)).
\end{equation}
We say $(\pi_i)_{i=1}^\infty$ is {\em quantum exchangeable} with respect to $\phi$ if for all $k\in\Nats$,
the $*$-homomorphisms $\pit_1,\ldots,\pit_k$ have the same distribution with respect to $\id\otimes\phi$
on $A_s(k)\otimes\Bfr$ as the $*$-homomorphisms $\pi_1,\ldots\pi_k$ do with respect to $\phi$ on $\Bfr$,
where
\begin{equation}\label{eq:pit}
\pit_i=\sum_{j=1}^ku_{i,j}\otimes\pi_j:A\to A_s(k)\otimes\Bfr.
\end{equation}
This means that for all $n\in\Nats$, all $i(1),\ldots,i(n)\in\{1\ldots,k\}$ and all $a_1,\ldots,a_n\in A$,
we have
\begin{equation}\label{eq:phipi}
\phi(\pi_{i(1)}(a_1)\cdots\pi_{i(n)}(a_n))1=(\id\otimes\phi)(\pit_{i(1)}(a_1)\cdots\pit_{i(n)}(a_n)).
\end{equation}
\end{defi}
\begin{remarks}
\begin{enumerate}[(a)]
\item In~\eqref{eq:pit}, the minimal tensor product $C^*$-algebra is indicated.
\item Looking at~\eqref{eq:exch}, we see that in Definition~\ref{def:exch}, it suffices to take $\sigma$ belonging
to the group $S_\infty$ of all permutations of $\Nats$ that fix all but finitely many elements.
\item
If~\eqref{eq:phipi} holds,
then composing with the standard ``abelianization'' representation $A_s(k)\to C(S_k)$,
we get~\eqref{eq:exch}.
Thus, quantum exchangeability implies exchangeability.
\end{enumerate}
\end{remarks}

Throughout this paper, we let
$\Afr=*_{i=1}^\infty A$
be the universal, unital free product of infinitely many copies of $A$ an let $\lambda_i:A\to\Afr$ ($i\in\Nats$)
denote the canonical injective $*$-homomorphisms arising from this construction.
(Thus, given any unital
$*$-homomorphisms $\pi_i$ from $A$ to a $C^*$-algebra $\Bfr$,
there is a unital $*$-homomorphism $\pi=\freeprod_{i=1}^\infty\pi_i:\Afr\to\Bfr$
such that $\pi\circ\lambda_i=\pi_i$.)
\begin{defi}\label{def:QSS}
A  state $\psi$ on $\Afr$ 
is called
{\em  symmetric} if $(\lambda_i)_{i=1}^\infty$ is exchangeable with respect to $\psi$ and
{\em quantum symmetric} if $(\lambda_i)_{i=1}^\infty$ is quantum exchangeable with respect to $\psi$.
Let $\SSt(A)$ be the set of all symmetric states on $\Afr$.
Let $\QSS(A)$ be the set of all quantum symmetric states on $\Afr$ and let $\TQSS(A)$ be the set of all quantum
symmetric states on $\Afr$ that are traces.
Furthermore, if $\phi$ is a state on $A$, let $\QSS(A,\phi)$ be the set of quantum symmetric states $\psi$ on $\Afr$ such that
$\psi\circ\lambda_i=\phi$ for some (and, thus, for all) $i$; if $\tau$ is a tracial state on $A$, let
$\TQSS(A,\tau)=\QSS(A,\tau)\cap\TQSS(A)$.
\end{defi}

{}From the compactness of the state space of $\Afr$ in the weak$^*$-topology, we immediately have the following:
\begin{prop}
For any unital $C^*$-algebra $A$ and any state $\phi$ on $A$ and any tracial state $\tau$ on $A$, the sets
$\SSt(A)$, $\QSS(A)$, $\TQSS(A)$, $\QSS(A,\phi)$ and $\TQSS(A,\tau)$ are compact, convex subsets of the Banach space dual of $\Afr$
in the weak$^*$-topology.
\end{prop}

\section{Quantum symmetric states arising from freeness}
\label{sec:freenessYields}

The following proposition says that if a state $\psi$ on $\Afr$
is equidistributed on the copies of $A$ and if it arises from a situation of freeness with amalgamation,
then it is quantum symmetric.
It is analogous to Prop. 3.1 of~\cite{KSp09}, and the proof is quite the same.
So we will outsource part of the proof, below, to~\cite{KSp09}.

\begin{prop}\label{prop:freeqsymm}
Let $B$ be a unital $C^*$-algebra, $D\subseteq B$ a unital $C^*$-subalgebra with
$E:B\to D$ a conditional expectation (i.e., a projection of norm one) onto $D$.
Suppose $\pi_i:A\to B$ ($i\in\Nats$) are $*$-homomorphisms
such that $E\circ\pi_i$ is the same for all $i$
and the family $(\pi_i(A))_{i=1}^\infty$ is free with respect to $E$.
Let $\pi=\freeprod_{i=1}^\infty\pi_i:\Afr\to B$ be the resulting free product $*$-homomorphism.
Let $\rho$ be any state of $D$ and consider the state $\psi=\rho\circ E\circ\pi$ of $\Afr$.
Then $\psi$ is quantum symmetric.
\end{prop}
\begin{proof}
Considering the state $\rhoh=\rho\circ E$ of $B$, and using $\pi\circ\lambda_i=\pi_i$, we must
show that $(\pi_i)_{i=1}^\infty$ is quantum exchangeable with respect to $\rhoh$.
Fixing $k,n\in\Nats$, $i(1),\ldots,i(n)\in\{1,\ldots,k\}$ and $a_1,\ldots,a_n\in A$ and working in
$A_s(k)\otimes B$, we have
\begin{align}
(\id\otimes\rhoh)&(\pit_{i(1)}(a_1)\cdots\pit_{i(n)}(a_n))= \label{eq:idrhoh} \\
&=\sum_{j(1),\ldots,j(n)=1}^ku_{i(1),j(1)}\cdots u_{i(n),j(n)}\cdot\rhoh(\pi_{i(1)}(a_1)\cdots\pi_{i(n)}(a_n))
 \notag \\ \displaybreak[2]
&=\sum_{j(1),\ldots,j(n)=1}^ku_{i(1),j(1)}\cdots u_{i(n),j(n)}\cdot\rho(E(\pi_{i(1)}(a_1)\cdots\pi_{i(n)}(a_n)))
 \notag \\ \displaybreak[2]
&=\sum_{j(1),\ldots,j(n)=1}^ku_{i(1),j(1)}\cdots u_{i(n),j(n)}\cdot
  \rho\bigg(\sum_{\sigma\in\NC(n)}\kappa_\sigma^E[\pi_{i(1)}(a_1),\ldots,\pi_{i(n)}(a_n)]\bigg) \notag \\ \displaybreak[2]
&=\sum_{\sigma\in\NC(n)}\sum_{j(1),\ldots,j(n)=1}^ku_{i(1),j(1)}\cdots u_{i(n),j(n)}\cdot
  \rho\big(\kappa_\sigma^E(\pi_{i(1)}(a_1),\ldots,\pi_{i(n)}(a_n))\big), \notag
\end{align}
where $\NC(n)$ denotes the lattice of noncrossing partitions of $\{1,\ldots,n\}$ and
$\kappa_\sigma^E$ denotes Speicher's free $D$-valued cummulant~\cite{Sp98} associated to $\sigma$
and the conditional expectation $E$ (see~\cite{KSp09} for more information about these).
Now fixing $\sigma\in\NC(n)$
and arguing as in the proof of Prop.\ 3.1 of~\cite{KSp09}, by the freeness assumption we have that
$\kappa_\sigma^E(\pi_{i(1)}(a_1),\ldots,\pi_{i(n)}(a_n))$ vanishes unless $\ker j\ge\sigma$,
where $j=(j(1),\ldots,j(n))$;
furthermore, by the assumption that $E\circ\pi_i$ is the same for all $i$,
it follows that the value of the cummulant $\kappa_\sigma^E(\pi_{i(1)}(a_1),\ldots,\pi_{i(n)}(a_n))$
is the same for all $j$ with $\ker j\ge\sigma$ and we denote this quantity simply by $\kappa_\sigma^E$.
At this point, the proof proceeds almost precisely as in the proof of Prop.\ 3.1 of~\cite{KSp09}
(starting at the last displayed equation on p.\ 480 of~\cite{KSp09}, but with
$\sigma$ replacing $\pi$ and $\rho$ or $\rhoh$ replacing $\phi$).
Thus, we obtain that the quantity~\eqref{eq:idrhoh} is equal to
\begin{align*}
\rho\big(\sum_{\substack{\sigma\in\NC(n) \\\ker i\ge\sigma}}\kappa_\sigma^E\big)
&=\rho\big(\sum_{\substack{\sigma\in\NC(n) \\\ker i\ge\sigma}}
 \kappa_\sigma^E[\pi_{i(1)}(a_1),\ldots,\pi_{i(n)}(a_n)]\big)= \\
&=\rho\big(E(\pi_{i(1)}(a_1),\ldots,\pi_{i(n)}(a_n))\big)
=\rhoh\big(\pi_{i(1)}(a_1)\cdots\pi_{i(n)}(a_n)\big),
\end{align*}
as required.
\end{proof}

\section{Two results on amalgmated free products of $C^*$-algebras}
\label{sec:Two}

In this section, we collect two technical results on an amalgamated free product of infinitely many copies of a $C^*$-algebra.
(This construction was introduced by Voiculescu in~\cite{V85};  see also the book~\cite{VDN92}.)
We let $B$ be a unital $C^*$-algebra, $D$ be a unital $C^*$-subalgebra of $B$,
$E:B\to D$ be a conditional expectation whose GNS representation is faithful (on $B$).
We let
\[
(A,F)=(*_D)_{i=1}^\infty(B,E)
\]
be the $C^*$-algebra free product with amalgamation of infinitely many copies of $(B,E)$,
and we denote by $B_i$ the $i$-th copy of $B$ in $A$ arising
from the free product construction.
We let $\rho$ be a state on $D$ and let $\pi=\pi_{\rho\circ F}$ denote the GNS representation of $A$ corresponding to the state $\rho\circ F$.

\begin{prop}\label{prop:rhoFt}
The formula $G(\pi(a))=\pi(F(a))$ defines a conditional expectation $G$ from $\pi(A)$ onto $\pi(D)$.
\end{prop}
\begin{proof}
Consider the dense $*$-subalgebra $A_0=\stalg(\bigcup_{i=1}^\infty B_i)$ of $A$.
We need only show $\|\pi(F(a))\|\le\|\pi(a)\|$ for all $a\in A_0$.
For this, it will suffice to show
\[
\big|\langle\pi(F(a))\vh,\wh\rangle\big|\le\|\pi(a)\|\,\|\vh\|_2\,\|\wh\|_2
\]
for all $v,w\in A_0$, where as usual, $v\mapsto\vh$ denotes the defining linear mapping $A\mapsto L^2(A,\rho\circ F)$
and $\|\cdot\|_2$ denotes the norm in $L^2(A,\rho\circ F)$.
Fixing such $v$ and $w$, let $N$ be so large that $v,w\in\stalg(\bigcup_{i=1}^N B_i)$.

For $\sigma\in S_\infty$, we have the $*$-automorphism $\beta_\sigma$ of $A$ that sends $B_i$ to the $B_{\sigma(i)}$ for all $i$, and
we have the $*$-endomorphism $\alpha$ of $A$ that sends $B_i$ to $B_{i+1}$ for all $i$.
Both $\beta_\sigma$ and $\alpha$ leave $F$ invariant.
By Theorem~6.1 of~\cite{AD09},
\[
F(a)=\lim_{n\to\infty}\frac1n\sum_{k=N+1}^{N+n}\alpha^k(a),
\]
where the limit is in norm.
So it will suffice to show
\[
\big|\langle\pi(\alpha^k(a))\vh,\wh\rangle\big|\le\|\pi(a)\|\,\|\vh\|_2\,\|\wh\|_2\qquad(k>N),
\]
where the inner product and last two norms are taken in $L^2(A,\rho\circ F)$.
But fixing $k>N$ and letting $\sigma\in S_\infty$ be such that $\beta_\sigma\circ\alpha^k(a)=a$, we have
\begin{multline*}
\big|\langle\pi(\alpha^k(a))\vh,\wh\rangle\big|
=|\rho\circ F(w^*\alpha^k(a)v)|
=|\rho\circ F(\beta_\sigma(w)^*a\beta_\sigma(v))| \\
=\big|\langle\pi(a))(\beta_\sigma(v))\hat{\;},(\beta_\sigma(w))\hat{\;}\rangle\big|
\le\|\pi(a)\|\,\|(\beta_\sigma(v))\hat{\;}\|_2\,\|(\beta_\sigma(w))\hat{\;}\|_2
=\|\pi(a)\|\,\|\vh\|_2\,\|\wh\|_2.
\end{multline*}
\end{proof}

\begin{prop}\label{prop:TinW*D}
Let
\[
\Tc=\bigcap_{N\ge1}W^*(\bigcup_{j\ge N}\pi(B_i)).
\]
Then $\Tc=W^*(\pi(D))$.
\end{prop}
\begin{proof}
The inclusion $\supseteq$ is clear.
For the reverse includion, let $z\in\Tc$.
We use the same notation as in the last proof.
Suppose $v_1,\ldots,v_p\in A_0$ satisfy $\|\vh_j\|_2\le1$ for all $j$.
Let $N\ge2$ be such that $v_1,\ldots,v_p\in\stalg(\bigcup_{i=1}^{N-1}B_i)$ and let $\eps>0$.
Using $z\in W^*(\bigcup_{j\ge N}(\pi(B_j)))$,
by Kaplansky's density theorem, there is $y\in\stalg(\bigcup_{i\ge N}B_i)$ such that $\|\pi(y)\|\le\|z\|$ and
$\|(z-\pi(y))\vh_j\|_2<\eps$ for all $j$.
By freeness, we have
$F(v_j^*yv_i)=F(v_j^*F(y)v_i)$
and this implies
\[
\langle\pi(y)\vh_i,\vh_j\rangle=\langle\pi(F(y))\vh_i,\vh_j\rangle,\qquad(1\le i,j\le p).
\]
Thus, letting $x=\pi(F(y))$ we have
\[
\big|\langle(z-x)\vh_i,\vh_j\rangle|<\eps,\qquad(1\le i,j\le p).
\]
By Proposition~\ref{prop:rhoFt}, we also have
\[
\|x\|=\|\pi(F(y))\|\le\|\pi(y)\|\le\|z\|.
\]
Finding in this way such an $x$ for every finite subset $\vh_1,\ldots\vh_p$ as above and every $\eps>0$, we get a net in $\pi(D)$
that converges in weak-operator topology to $z$.
(The inequalities $\|x\|\le\|z\|$ allow us to take vectors from only the dense subspace of $L^2(A,\rho\circ F)$ spanned by the $\vh$.)
Thus, $z\in W^*(\pi(D))$.
\end{proof}

\begin{remark}
It is natural to ask:  how much more general is the property demonstrated in Proposition~\ref{prop:rhoFt}?
For example, let $D$ be a unital $C^*$-subalgebra of a $C^*$-algebra $A$
and suppose $F:A\to D$ is a conditional expectation whose GNS representation is faithful;
let $\rho$ be a state of $D$ and let $\pi=\pi_{\rho\circ F}$ be the GNS representation of the state $\rho\circ F$ of $A$;
let $x\in A$.
Must we have $\|\pi(F(x))\|\le\|\pi(x)\|$?
The following easy example shows that the answer is ``no.''
Take
\begin{align*}
D=\Cpx\oplus\Cpx\oplus\Cpx&\;\hookrightarrow\;M_2(\Cpx)\oplus M_2(\Cpx)=A \\
 a\oplus b\oplus c&\;\mapsto\;
\left(\begin{matrix}a&0\\0&b\end{matrix}\right)\oplus\left(\begin{matrix}b&0\\0&c\end{matrix}\right)
\end{align*}
and let $F:A\to D$ be the conditional expectation
\[
\left(\begin{matrix}a_{11}&a_{12}\\a_{21}&a_{22}\end{matrix}\right)
 \oplus\left(\begin{matrix}b_{11}&b_{12}\\b_{21}&b_{22}\end{matrix}\right)
\mapsto a_{11}\oplus\left(\frac{a_{22}+b_{11}}2\right)\oplus b_{22}.
\]
Let $\rho$ be the state on $D$ given by $\rho(a\oplus b\oplus c)=a$.
Then
\[
\ker(\pi_{\rho\circ F})=0\oplus M_2(\Cpx),
\]
and identifying $D$ with its embedded image in $A$, we have
\[
\ker\pi_{\rho\circ F}\ni
\left(\begin{matrix}0&0\\0&0\end{matrix}\right)\oplus\left(\begin{matrix}2&0\\0&0\end{matrix}\right)
\;\overset{F}{\mapsto}\;
\left(\begin{matrix}0&0\\0&1\end{matrix}\right)\oplus\left(\begin{matrix}1&0\\0&0\end{matrix}\right)
\notin\ker\pi_{\rho\circ F}.
\]
\end{remark}

\section{Tail algebras of symmetric states}
\label{sec:tails}

This section has four parts and is about general symmetric states.

\subsection{The tail algebra and a conditional expectation}
\label{subsec:tail}

Let $\psi$ be a symmetric state on $\Afr$.
Let $\pi_\psi:\Afr\to B(L^2(\Afr,\psi))$ be the GNS representation associated to $\psi$.
Thus, $\HEu_\psi=L^2(\Afr,\psi)$ is the Hilbert space and, as is usual,
$a\mapsto\ah$ denotes the linear mapping onto a dense subspace of $L^2(\Afr,\psi)$,
with $\langle \ah,\bh\rangle=\psi(b^*a)$ for all $a,b\in\Afr$, norm $\|\ah\|_2=\psi(a^*a)^{1/2}$
and $\pi_\psi(a)\bh=(ab)\hat{\:}$.
Let $\Mcal_\psi=\pi_\psi(\Afr)''$ be the von Neumann algebra generated by the image of $\pi_\psi$.
Consider the vector state $\psih=\langle\cdot\oneh,\oneh\rangle$ on $\Mcal_\psi$;
we have $\psih\circ\pi_\psi=\psi$.

\begin{defisub}\label{def:tail}
The {\em tail algebra} of $\psi$ is the von Neumann subalgebra
\[
\Tc_\psi=\bigcap_{N=1}^\infty W^*\big(\bigcup_{i=N}^\infty\pi_\psi\circ\lambda_i(A)\big)
\]
of $\Mcal_\psi$.
\end{defisub}

\begin{defisub}\label{defi:fixedpt}
Let $\psi$ be a symmetric state on $\Afr$.
With $S_\infty=\bigcup_{n=1}^\infty S_n$ denoting the group of finite permutations of $\Nats$, we have the $\psi$-preserving
action $S_\infty\ni\sigma\mapsto\beta_\sigma\in\Aut(\Afr)$
defined by $\beta_\sigma\circ\lambda_i=\lambda_{\sigma(i)}$.
This yields a unitary $U_\sigma$ on $\HEu_\psi$ given by $U_\sigma\xh=(\beta_\sigma(x))\hat{\;}$.
Conjugation with $U_\sigma$ yields an automorphism $\betah_\sigma$ of $\Mcal_\psi$ such that
$\sigma\mapsto\betah_\sigma$ is an action of $S_\infty$ on $\Mcal_\psi$
and
\begin{equation}\label{eq:betahprops}
\betah_\sigma\circ\pi_\psi=\pi_\psi\circ\beta_\sigma,\qquad
\psih\circ\betah_\sigma=\psih.
\end{equation}
The {\em fixed point algebra} of $\psi$ is the fixed point subalgebra $\Fc_\psi=\Mcal_\psi^\betah$ of this action.
It is a von Neumann subalgebra of $\Mcal_\psi$.
\end{defisub}

\begin{lemmasub}\label{lem:TinD}
Let $\psi$ be a symmetric state on $\Afr$.
Then $\Tc_\psi\subseteq\Fc_\psi$.
\end{lemmasub}
\begin{proof}
Let $\sigma\in S_n\subset S_\infty$.
Then $\betah_\sigma(x)=x$ for all $x\in W^*(\bigcup_{j>n}\pi_j(A))$.
This implies $\betah_\sigma(x)=x$ for all $x\in\Tc_\psi$.
\end{proof}

There are many cases where $\Tc_\psi=\Fc_\psi$ (see Example~\ref{ex:Liu} and Proposition~\ref{prop:N=M}, below).
However, the example described in the next proposition shows that equality need not always hold,
even when $\psi$ is quantum symmetric.

\begin{examplesub}\label{ex:TneD}
Let $A=M_2(\Cpx)$ and let $\phi$ be the non-faithful state on $A$ given by
\[
\phi(\left(\begin{smallmatrix}a_{11}&a_{12}\\a_{21}&a_{22}\end{smallmatrix}\right))=a_{11}.
\]
Let $\psi=*_1^\infty\phi:\Afr\to\Cpx$ be the free product state of $\phi$ with itself infinitely many times.
(Thus, by Proposition~\ref{prop:freeqsymm}, $\psi$ is quantum symmetric.)
We will show that the tail algebra $\Tc_\psi$ is trivial, i.e., $\Tc_\psi=\Cpx1$, while $\Fc_\psi\ne\Cpx1$.

Letting $(e_{ij})_{1\le i,j\le 2}$ be the usual system of matrix units,
$L^2(A,\phi)$ has orthonormal basis $\{\eh_{11},\eh_{21}\}$ with $\oneh=\eh_{11}$.
Realizing $\HEu_\psi=L^2(\Afr,\psi)$ as the free product of Hilbert spaces, we naturally see that $L^2(\Afr,\psi)$ has orthonormal basis
\[
\{\xi\}\cup\{\eh_{21}^{(p_1)}\otimes\cdots\otimes\eh_{21}^{(p_n)}\mid n,p_1,\ldots,p_n\in\Nats,\,p_j\ne p_{j+1}\},
\]
where $\eh_{21}^{(p)}$ represents the element $\eh_{21}$ in the $p$-th copy of $L^2(A,\phi)$.
Letting $e_{ij}^{(p)}=\pi_\psi\circ\lambda_p(e_{ij})$, by the free product construction we have that
$e_{22}^{(p)}$ is the projection onto 
\[
\lspan\{\eh_{21}^{(p_1)}\otimes\cdots\otimes\eh_{21}^{(p_n)}\mid p_1=p,\,n,p_2,\ldots,p_n\in\Nats,\,p_j\ne p_{j+1}\}.
\]

We see that the sum $\sum_{p=1}^\infty e_{22}^{(p)}$ converges in strong-operator-topology to $1-P_\xi$, where $P_\xi$ is the rank one
projection onto $\Cpx\xi$.
So $P_\xi\in\Mcal_\psi$.
Since $\xi$ is cyclic for the representation $\pi_\psi$, we find
\begin{equation}\label{eq:MpsiBL2}
\Mcal_\psi=B(\HEu_\psi).
\end{equation}

Considering now the action of $\sigma\in S_\infty$, the unitary $U_\sigma$ leaves $\xi$ fixed and satisfies
\[
U_\sigma:\eh_{21}^{(p_1)}\otimes\cdots\otimes\eh_{21}^{(p_n)}\mapsto\eh_{21}^{(\sigma(p_1))}\otimes\cdots\otimes\eh_{21}^{(\sigma(p_n))}.
\]
Thus,
\[
P_\xi\in\Fc_\psi.
\]
In particular, $\Fc_\psi$ is nontrivial.

An application of Proposition~\ref{prop:TinW*D} shows$\Tc_\psi=\Cpx1$.
Indeed, the GNS representation $\pi_\psi$ maps $\Afr$ onto the free product $C^*$-algebra
$*_1^\infty(A,\phi)$, with amalgamation over the scalars.
So Proposition~\ref{prop:TinW*D} implies $\Tc_\psi\subseteq W^*(\pi_\psi(\Cpx))=\Cpx1$.
\end{examplesub}

We continue, letting $\psi$ be an arbitrary symmetric state on $\Afr$, with associated constructions as described above.
The next lemma shows that exchangeability passes from $\psi$ on $\Afr$ to $\psih$, with $\Fc_\psi$ included too.
\begin{lemmasub}\label{lem:R0exch}
If $n,i_1,\ldots,i_n\in\Nats$, $a_1,\ldots,a_n\in A$, $d_0,d_1,\ldots,d_n\in\Fc_\psi$ and $\sigma\in S_\infty$, then
\[
\psih(d_0\pi_{i_1}(a_1)d_1\cdots\pi_{i_n}(a_n)d_n)=
\psih(d_0\pi_{\sigma(i_1)}(a_1)d_1\cdots\pi_{\sigma(i_n)}(a_n)d_n).
\]
\end{lemmasub}
\begin{proof}
This follows immediately from the properties~\eqref{eq:betahprops} of $\betah$.
\end{proof}

Let
\[
\REu_0=\stalg(\Fc_\psi\cup\pi_\psi(\Afr_0)).
\]
We let $\alpha\in\End(\Afr)$ be the shift $*$-endomorphism, defined by $\alpha\circ\pi_i=\pi_{i+1}$.
\begin{lemmasub}\label{lem:alpha0}
There is a $*$-endomorphism $\alpha_0:\REu_0\to\REu_0$ satisfying
\begin{enumerate}[(i)]
\item\label{lem:alpha0-1} $\alpha_0\circ\pi_i=\pi_{i+1}$ ($i\in\Nats$)
\item\label{lem:alpha0-2} $\forall x\in\REu_0$, $\|\alpha_0(x)\|=\|x\|$
\item\label{lem:alpha0-3} $\alpha_0\restrict_{\Fc_\psi}=\id_{\Fc_\psi}$
\item\label{lem:alpha0-4} $\psih\circ\alpha_0=\psih\restrict_{\REu_0}$.
\end{enumerate}
\end{lemmasub}
\begin{proof}
Taking the $m$-cycle $c_m=(1,2,\ldots,m)\in S_\infty$, we easily see that for all $x\in\REu_0$, $\beta_{c_m}(x)$ is
constant in $m$ for $m$ large enough, and we let
\begin{equation}\label{eq:alpha0}
\alpha_0(x)=\lim_{m\to\infty}\betah_{c_m}(x)=\betah_{c_M}(x)
\end{equation}
for $M$ large depending on $x$.
Clearly, $\alpha_0$ is a $*$-endomorphism and \eqref{lem:alpha0-1}-\eqref{lem:alpha0-4} hold.
\end{proof}

\begin{lemmasub}\label{lem:alphabeta}
Let $\sigma\in S_k$ and take $n\ge k$.
Then $\betah_\sigma\circ\alpha_0^n=\alpha_0^n$.
\end{lemmasub}
\begin{proof}
The image of $\alpha_0^n$ lies in the $*$-algebra generated by $\Fc_\psi\cup\bigcup_{j>n}\pi_j(A)$,
which, under the hypothesis $n\ge k$, lies in the fixed point algebra of the automorphism $\betah_\sigma$.
\end{proof}

\begin{lemmasub}\label{lem:alphanconst}
For all $x\in\REu_0$, all $y,z\in\Afr_0$ and all $\sigma\in S_\infty$, there is $N\in\Nats$ and there is $w\in\Cpx$ such that
for all $n\ge N$, we have
\[
\langle\alpha_0^n(x)\yh,\zh\rangle=\langle\alpha_0^n(\betah_\sigma(x))\yh,\zh\rangle=\langle\betah_\sigma(\alpha_0^n(x))\yh,\zh\rangle=w.
\]
\end{lemmasub}
\begin{proof}
We may without loss of generality assume $x$, $y$ and $z$ are monomials, i.e.,
\begin{align}
x&=d_0\pi_{i_1}(a_1)d_1\cdots\pi_{i_p}(a_p)d_p \label{eq:xmon} \\
y&=\pi_{j_1}(a_1')\pi_{j_2}(a_2')\cdots\pi_{j_q}(a_q') \notag \\
z&=\pi_{k_1}(a_1'')\pi_{k_2}(a_2'')\cdots\pi_{k_r}(a_r'') \notag
\end{align}
for some $p,q,r\in\Nats$, $d_0,\ldots,d_p\in\Fc_\psi$ and (for all indices $f$, $g$ and $h$ in the appropriate ranges)
$i_f,j_g,k_h\in\Nats$ and $a_f,a_g',a_h''\in A$.
Then we have
\[
\langle\alpha_0^n(x)\yh,\zh\rangle=
\psih\big(\pi_{k_r}(a_r'')^*\cdots\pi_{k_1}(a_1'')^*\;d_0\pi_{i_1+n}(a_1)d_1\cdots\pi_{i_p+n}(a_p)d_p\;
\pi_{j_1}(a_1')\cdots\pi_{j_q}(a_q')\big).
\]
From the exchangeability result of Lemma \ref{lem:R0exch}, we see that this quantity is independent of $n$ so long as
\begin{equation}\label{eq:nbig}
n\ge\max(j_1,\ldots,j_q,k_1,\ldots,k_r).
\end{equation}
Let $w$ be the common value of $\langle\alpha_0^n(x)\yh,\zh\rangle$ for $n$ satisfying~\eqref{eq:nbig}.
We also have
\begin{multline*}
\langle\alpha_0^n(\betah_\sigma(x))\yh,\zh\rangle \\
=\psih\big(\pi_{k_r}(a_r'')^*\cdots\pi_{k_1}(a_1'')^*\;d_0\pi_{\sigma(i_1)+n}(a_1)d_1\cdots\pi_{\sigma(i_p)+n}(a_p)d_p\;
\pi_{j_1}(a_1')\cdots\pi_{j_q}(a_q')\big)
\end{multline*}
and again invoking exchangeability, we get $\langle\alpha_0^n(\betah_\sigma(x))\yh,\zh\rangle=w$ for $n$ satisfying~\eqref{eq:nbig}.
Finally, the equality
\[
\langle\alpha_0^n(x)\yh,\zh\rangle=\langle\betah_\sigma(\alpha_0^n(x))\yh,\zh\rangle
\]
for all $n$ large enough follows from Lemma~\ref{lem:alphabeta}.
\end{proof}

Let
\[
\QEu_0=\stalg(\Tc_\psi\cup\pi_\psi(\Afr_0))
\]
and note that the endomorphism $\alpha_0$ of $\REu_0$ maps $\QEu_0$ into $\QEu_0$.
Consider the weakly dense $C^*$-subalgebras
\[
\REu_\psi=\overline{\REu_0}^{\|\cdot\|},\qquad\QEu_\psi=\overline{\QEu_0}^{\|\cdot\|}
\]
of $\Mcal_\psi$.
From Lemmas~\ref{lem:alpha0} and~\ref{lem:alphabeta}, we immediately have the following.
\begin{lemmasub}\label{lem:alpha}
There is a unique $*$-endomorphism $\alpha$ of $\REu_\psi$ satisfying
\begin{enumerate}[(i)]
\item\label{lem:alpha-1} $\alpha\circ\pi_i=\pi_{i+1}$ ($i\in\Nats$)
\item\label{lem:alpha-2} $\alpha\restrict_{\Fc_\psi}=\id_{\Fc_\psi}$
\end{enumerate}
Moreover, we have
\begin{enumerate}[(i)]
\setcounter{enumi}{2}
\item\label{lem:alpha-3} $\forall x\in\REu_\psi$, $\|\alpha(x)\|=\|x\|$
\item\label{lem:alpha-4} $\psih\circ\alpha=\psih\restrict_{\REu_\psi}$
\item\label{lem:alpha-5} $\alpha(\QEu_\psi)\subseteq\QEu_\psi$
\item\label{lem:alpha-6} if $\sigma\in S_k$ and $n\ge k$, then $\betah_\sigma\circ\alpha^n=\alpha^n$.
\end{enumerate}
\end{lemmasub}

\begin{thmsub}\label{thm:E}
There is a conditional expectation $G_\psi$ (i.e., a projection of norm~$1$) from $\REu_\psi$ onto $\Fc_\psi$ given by
\begin{equation}\label{eq:WOTlim}
G_\psi(x)={\rm WOT-}\lim_{n\to\infty}\alpha^n(x),
\end{equation}
where the above limit is in weak operator topology.
This satisfies the following:
\begin{enumerate}[(i)]
\item\label{thm:E-1} $\psih\circ G_\psi=\psih\restrict_{\REu_\psi}$
\item\label{thm:E-2} $G_\psi\circ\alpha=G_\psi$
\item\label{thm:E-3} for all $\sigma\in S_\infty$, $G_\psi\circ\betah_\sigma\restrict_{\REu_\psi}=G_\psi$
\item\label{thm:E-4} the restriction $F_\psi$ of $G_\psi$ to $\QEu_\psi$ is a conditional expectation from $\QEu_\psi$ onto $\Tc_\psi$.
\end{enumerate}
\end{thmsub}
\begin{proof}
Let $x\in\REu_\psi$ and $\xi,\eta\in\HEu_\psi$ and take $\eps>0$.
Take $x_0\in\REu_0$ such that $\|x-x_0\|<\eps$ and $y,z\in\Afr_0$ such that $\|\xi-\yh\|_2<\eps$ and $\|\eta-\zh\|_2<\eps$.
Then, for every $n\in\Nats$,
\begin{align}
|\langle\alpha^n(x)\xi,\eta&\rangle-\langle\alpha^n(x_0)\yh,\zh\rangle| \label{eq:xxe} \\
&\le|\langle\alpha^n(x-x_0)\xi,\eta\rangle|
  +|\langle\alpha^n(x_0)(\xi-\yh),\eta\rangle|
  +|\langle\alpha^n(x_0)\yh,\eta-\zh\rangle| \notag \\
&\le \eps\|\xi\|_2\|\eta\|_2 + \eps(\|x\|+\eps)\|\eta\|_2+\eps(\|x\|+\eps)(\|\xi\|_2+\eps). \notag
\end{align}
But by Lemma~\ref{lem:alphanconst}, $\langle\alpha^n(x_0)\yh,\zh\rangle$ is constant for large enough $n$.
We deduce that the quantity $\langle\alpha^n(x)\xi,\eta\rangle$ is Cauchy in $n$, and, therefore, the limit
on the right-hand-side of \eqref{eq:WOTlim} exists.
Hence, \eqref{eq:WOTlim} defines $G_\psi$ as a linear function from $\REu_\psi$ into $\Mcal_\psi$.
Since $\alpha$ is isometric, $G_\psi$ is contractive.

Let $\sigma\in S_\infty$.
We have the same upper bound for the quantity
\[
|\langle\betah_\sigma\circ\alpha^n(x)\xi,\eta\rangle-\langle\betah_\sigma\circ\alpha^n(x_0)\yh,\zh\rangle|
\]
as in~\eqref{eq:xxe}.
But by Lemma~\ref{lem:alphanconst}, $\langle\betah_\sigma\alpha^n(x_0)\yh,\zh\rangle=\langle\alpha^n(x_0)\yh,\zh\rangle$
for all $n$ sufficiently large.
This implies $\betah_\sigma\circ G_\psi(x)=G_\psi(x)$.
So $G_\psi(\REu_\psi)\subseteq\Fc_\psi$.

On the other hand, if $x\in\Fc_\psi$, then from Lemma~\ref{lem:alpha}\eqref{lem:alpha-2} we deduce $G_\psi(x)=x$.
Thus, $G_\psi$ is a projection of norm $1$ from $\REu_\psi$ onto $\Fc_\psi$.
Now properrty~\eqref{thm:E-1} follows from Lemma~\ref{lem:alpha}\eqref{lem:alpha-4}.
while property~\eqref{thm:E-2} is apparent from the definition.

For~\eqref{thm:E-3}, it will suffice to show
\[
\langle G_\psi(\betah_\sigma(x))\yh,\zh\rangle=\langle G_\psi(x)\yh,\zh\rangle
\]
for every $x\in\REu_0$ and $y,z\in\Afr_0$.
However, this follows because Lemma~\ref{lem:alphanconst} shows
\[
\langle\alpha_0^n(\betah_\sigma(x))\yh,\zh\rangle=\langle\alpha_0^n(x)\yh,\zh\rangle
\]
for all $n$ sufficiently large.

To prove~\eqref{thm:E-4} we claim that it will suffice to show $G_\psi(x)\in\Tc_\psi$ for every $x\in\QEu_0$.
Indeed, because $G_\psi$ is contractive and $\QEu_0$ is norm dense
in $\QEu_\psi$, it will follow that $G_\psi(\QEu_\psi)\subseteq\Tc_\psi$;
moreover, since $\Tc_\psi\subseteq\QEu_\psi$ and $G_\psi$ is the identity on $\Tc_\psi$, we will also have that $G_\psi(\QEu_\psi)=\Tc_\psi$.
We may without loss of generality assume $x$ is a monomial of the form~\eqref{eq:xmon} but with $d_0,\ldots,d_p\in\Tc_\psi$.
Now
\[
\alpha^n(x)=d_0\pi_{i_1+n}(a_1)d_1\cdots\pi_{i_p+n}(a_p)d_p\in\stalg\left(\Tc_\psi\cup \bigcup_{j>n}\pi_j(A)\right),
\]
so
\[
G_\psi(x)\in W^*\left(\Tc_\psi\cup\bigcup_{j>N}\pi_j(A)\right)\subseteq W^*\left(\,\bigcup_{j>N}\pi_j(A)\right)
\]
for every $N\in\Nats$.
Thus, $G_\psi(x)\in\Tc_\psi$.
\end{proof}

\begin{remarksub}\label{rem:WOTconst}
We record here for future reference the fact that follows from the exchangeability result of Lemma~\ref{lem:R0exch},
as used in the above proof:
for all $x,y,z\in\REu_0$,
there is $N$ sufficiently large that for all $n\ge N$, we have
\[
\psih(z^*\alpha^n(x)y)=\psih(z^*G_\psi(x)y).
\]
\end{remarksub}

As an immediate consequence of property~\eqref{thm:E-3} of the above theorem, we have
that $G_\psi$ is exchangeable, also with elements of $\Fc_\psi$ added in:
\begin{corsub}\label{cor:Fexch}
For all $n,i_1,\ldots,i_n\in\Nats$, $a_1,\ldots,a_n\in A$, $d_0,\ldots,d_n\in\Fc_\psi$ and all $\sigma\in S_\infty$
\[
G_\psi(d_0\pi_{i_1}(a_1)d_1\cdots\pi_{i_n}(a_n)d_n)=G_\psi(d_0\pi_{\sigma(i_1)}(a_1)d_1\cdots\pi_{\sigma(i_n)}(a_n)d_n).
\]
\end{corsub}

\medskip

\subsection{A normal conditional expectation}
\label{subsec:normal}

To summarize the earlier notation and constructions:
\renewcommand{\labelitemi}{$\bullet$}
\begin{itemize}
\item $\psi$ is a symmetric state on $\Afr=*_{i=1}^\infty A$,
\item $\pi_\psi$ is the GNS representation of $\psi$ on $\HEu_\psi=L^2(\Afr,\psi)$,
\item $\Mcal_\psi$ is the von Neumann algebra generated by $\pi_\psi(\Afr)$,
\item $\Fc_\psi$ is the fixed point algebra of the permutation action on $\Mcal_\psi$,
\item $\Tc_\psi\subseteq\Fc_\psi$ is the tail algebra of $\psi$,
\item $\REu_\psi$ is the $C^*$-algebra generated by $\pi_\psi(\Afr)\cup\Fc_\psi$,
\item $\QEu_\psi$ is the $C^*$-algebra generated by $\pi_\psi(\Afr)\cup\Tc_\psi$,
\item $G_\psi:\REu_\psi\to\Fc_\psi$ is the conditional expectation onto $\Fc_\psi$ constructed in Theorem~\ref{thm:E}.
\item $F_\psi:\QEu_\psi\to\Tc_\psi$ is the conditional expectation onto $\Tc_\psi$ that is the restriction of $G_\psi$.
\end{itemize}

The following example is due to Weihua Liu~\cite{Liu}.
It has $\Fc_\psi=\Tc_\psi$ and shows that the conditional expectation
$F_\psi$ need not have an extension to a normal conditional expectation $\Mcal_\psi\to\Tc_\psi$.
\begin{examplesub}\label{ex:Liu}
Let $A=\Cpx^3$ and let $a=1\oplus-1\oplus0\in A$.
Let $\HEu$ be a Hilbert space with orthonormal basis $(v_n)_{n\ge0}$ and with corresponding system of matrix units 
$(e_{ij})_{i,j\ge0}$ that densely span the compact operators $K(\HEu)$ on $\HEu$.
For every $n\in\Nats$
let $\pi_n:A\to B(\HEu)$ be the unital $*$-homomorphism so that
$\pi_n(a)=e_{0n}+e_{n0}$.

Let $\pi=*_{n=1}^\infty\pi_n:\Afr\to B(\HEu)$ be the unital free product representation.
The $C^*$-algebra generated by $\pi(\Afr)$ is easily seen to be $K(\HEu)+\Cpx1$, where $K(\HEu)$ is the algebra
of compact operators on $\HEu$.

Let $\psi$ be the state on $\Afr$ that is the composition of $\pi$ and the vector state $\langle\cdot\,v_0,v_0\rangle$.
Let $S_\infty\ni\sigma\mapsto U_\sigma$ be the unitary representation given by
\[
U_\sigma(v_j)=\begin{cases}
v_{\sigma(j)},&j>0 \\
v_0,&j=0.
\end{cases}
\]
Then $U_\sigma\pi_n(\cdot)U_\sigma^*=\pi_{\sigma(n)}$.
Since $U_\sigma v_0=v_0$ for all $\sigma$, we see that $\psi$ is a symmetric state on $\Afr$.
Moveover, $v_0$ is cyclic under the action of $\pi(\Afr)$, so that $(\pi,\HEu,v_0)$ is equivalent to the GNS representation $(\pi_\psi,\HEu_\psi,\oneh)$.

The action $\betah$ of $S_\infty$ is given by $\beta_\sigma=\Ad(U_\sigma)$ and the fixed point algebra $\Fc_\psi$ is
the set of all operators that commute with all $U_\sigma$.
We see that $\Fc_\psi$ is the two-dimensional $C^*$-algebra $\Cpx e_{00}+\Cpx(1-e_{00})$.
Since $e_{00}=\pi_n(a^2)\pi_{n+1}(a^2)$ for every $n\ge1$, we see that the tail algebra, $\Tc_\psi$, contains $e_{00}$ and is equal to $\Fc_\psi$.
In particular, $\Tc_\psi=\Fc_\psi\subseteq\pi_\psi(\Afr)$ and $\QEu_\psi=\REu_\psi=\pi_\psi(\Afr)$.

The conditional expectation $G_\psi:\REu_\psi\to\Fc_\psi$ is given by
\[
G_\psi(x)=e_{00}xe_{00}+\tau(x)(1-e_{00}),
\]
where $\tau:\REu_\psi\to\Cpx$ is the character of $\REu_\psi$.
The von Neumann algebra $\Mcal_\psi$ is generated by $\REu_\psi$ and, thus, is equal to all of $B(\HEu)$.
Note that $e_{00}+\cdots+e_{nn}$ is increasing and converges in strong-operator topology to $1$ as $n\to\infty$,
but that $G_\psi(e_{00}+\cdots+e_{nn})=e_{00}$ for all $n$ while $G_\psi(1)=1$.
Thus, we see that $G_\psi$ has no extension to a normal conditional expectation from $B(\HEu)$ to $\Fc_\psi$.
\end{examplesub}

\medskip
We now build on the constructions performed in Section~\ref{subsec:tail}
to construct normal conditional expectations from possibly larger von Neumann
algebras onto the fixed point algebra $\Fc_\psi$ and the tail algebra $\Tc_\psi$.
For the most part, we will be more interested in the tail algebra than in the fixed point algebra, and our exposition will reflect this.

\begin{propsub}\label{prop:Npsi}
Given a symmetric state $\psi$ on $\Afr=*_1^\infty A$, there is a von Neumann algebra $\Nc_\psi$ and an injective, unital $*$-homomorphism
$\sigma_\psi:\QEu_\psi\to\Nc_\psi$ such that
\begin{enumerate}[(i)]
\item\label{it:Npsi-1} $\Nc_\psi$ is generated as a von Neumann algebra by $\sigma_\psi(\QEu_\psi)$,
\item\label{it:Npsi-2} the restriction of $\sigma_\psi$ to $\Tc_\psi\subseteq\QEu_\psi$ is an injective, normal $*$-homomorphism,
\item\label{it:Npsi-3} upon identifying $\QEu_\psi$ and $\Tc_\psi$ with their images under $\sigma_\psi$,
the conditional expectation $F_\psi:\QEu_\psi\to\Tc_\psi$ extends to a normal conditional expectation $E_\psi:\Nc_\psi\to\Tc_\psi$,
whose GNS representation is faithful on $\Nc_\psi$,
\item\label{it:Npsi-4} the image of $\Nc_\psi$
under the GNS representation $\pi_{\psih\circ E_\psi}$ of the normal state $\psih\restrict_{\Tc_\psi}\circ E_\psi$
is the von Neumann algebra $\Mcal_\psi$, so $\Mcal_\psi$ may be identified with $R\Nc_\psi$ for a central projection $R$ of $\Nc_\psi$
and the GNS representation $\pi_{\psih\circ E_\psi}$ itself with $\Nc_\psi\ni x\mapsto Rx$.
\end{enumerate}
\end{propsub}
\begin{proof}
Let $\pi_{F_\psi}$ denote the GNS representation of $\QEu_\psi$ on the right Hilbert $\Tc_\psi$-module $L^2(\QEu_\psi,F_\psi)$.
Since $\psih\restrict_{\QEu_\psi}=\psih\circ F_\psi$ and since the GNS representation of $\psih$ is faithful on $\QEu_\psi$,
we see that $\pi_{F_\psi}$ is faithful.
Moreover, the projection $F_\psi$ splits off a copy of $\Tc_\psi$ from $L^2(\QEu_\psi,F_\psi)$ as a complemented submodule:
\begin{equation}\label{eq:Tsummand}
L^2(\QEu_\psi,F_\psi)=L^2(\QEu_\psi,F_\psi)\oup\oplus\Tc_\psi.
\end{equation}
Let $\kappa$ be a normal, faithful $*$-representation of $\Tc_\psi$ on a Hilbert space $\Vc$.
Let
\begin{equation}\label{eq:W}
\Wc=L^2(\QEu_\psi,F_\psi)\otimes_\kappa\Vc
\end{equation}
be the interior tensor product of Hilbert modules (see, e.g., \cite{L95});
thus, $\Wc$ is a Hilbert space and the action $\pi_{F_\psi}$ of $\QEu_\psi$ on $L^2(\QEu_\psi,F_\psi)$ gives rise to a $*$-representation
$\sigma_\psi:=\pi_{F_\psi}\otimes1$ of $\QEu_\psi$ on $\Wc$, given by $\sigma_\psi(x)=\pi_{F_\psi}(x)\otimes\id_\Vc$.
We let $\Nc_\psi$ denote the von Neumann algebra generated by $\sigma_\psi(\QEu_\psi)$.
The Hilbert module projection $F_\psi$ that projects onto the summand $\Tc_\psi$ in~\eqref{eq:Tsummand} induces the Hilbert space projection
$P=F_\psi\otimes 1$ on $\Wc$, whose range space is $\Tc_\psi\otimes_\kappa\Vc$, which is just a copy of $\Vc$.
The action of $\Tc_\psi$ on $\Wc$ via $\pi_{F_\psi}\otimes1$ commutes with the projection $P$ and we have
\[
\sigma_\psi(x)P=\kappa(x),\qquad(x\in\Tc_\psi).
\]
Moreover, we have
\[
P\sigma_\psi(x)P=\sigma_\psi(F_\psi(x))P\qquad(x\in\QEu_\psi).
\]
Since $\kappa$ is faithful, we may identify $\Tc_\psi\subseteq\QEu_\psi$ with $\sigma_\psi(\Tc_\psi)P$ and, doing so,
we find that the compression by $P$ implements the conditional expectation $F_\psi:\QEu_\psi\to\Tc_\psi$.
Thus, compression by $P$ realizes a normal conditional expectation
\[
E_\psi:\Nc_\psi\to\Tc_\psi
\]
that extends the conditional expectation $F_\psi$.

The vector $\oneh\in L^2(\QEu_\psi,F_\psi)$ is cyclic for the left action of $\QEu_\psi$, we see that the image of the projection $P$ is cyclic for the action
of $\Nc_\psi$ on $\Wc$, in the sense that $\lspan\bigcup_{x\in\Nc_\psi}xP(\Wc)$ is dense in $\Wc$.
Thus, the GNS representation of the conditional expectation $E_\psi$ is faithful on $\Nc_\psi$.

Since any two normal, faithful $*$-representations of a von Neumann algebra $\Tc_\psi$ can be obtained, one from the other,
by (a) inducing, i.e., tensoring with the identity on an infinite dimensional Hilbert space and (b) reducing, i.e., compressing by a projection
in the commutant, and since these operations when performed on $\kappa$ above carry over to the representation of $\Nc_\psi$,
we find that the constructions of the von Neumann algebra $\Nc_\psi$ and of the normal conditional expectation $E_\psi$ are independent of
the choice of $\kappa$.

The composition $\psih\restrict_{\Tc_\psi}\circ E_\psi$
is a normal state on $\Nc_\psi$, and the composition $\psih\restrict_{\Tc_\psi}\circ E_\psi\circ\sigma_\psi\circ\pi_\psi$ equals the original state
$\psi$ on $\QEu_\psi$.
Thus, we find that $\Mcal_\psi$ is the image of $\Nc_\psi$ under the GNS representation of $\psih\restrict_{\Tc_\psi}\circ E_\psi$ and may, therefore,
be identified with $R\Nc_\psi$ for some central projection $R\in Z(\Nc_\psi)$, and the GNS representation itself with the map
$\Nc_\psi\ni x\mapsto Rx$.
\end{proof}

A precisely analogous result with $\Fc_\psi$ and $\REu_\psi$ replacing $\Tc_\psi$ and $\QEu_\psi$, respectively, can be
obtained by repeating the above construction,
{\em mutatis mutandis}.
We now state this result in order to introduce notation that will be needed below.
\begin{propsub}\label{prop:Ppsi}
Given a symmetric state $\psi$ on $\Afr=*_1^\infty A$, there is a von Neumann algebra $\Pc_\psi$ and an injective, unital $*$-homomorphism
$\eta_\psi:\REu_\psi\to\Pc_\psi$ such that
\begin{enumerate}[(i)]
\item\label{it:Ppsi-1} $\Pc_\psi$ is generated as a von Neumann algebra by $\eta_\psi(\REu_\psi)$,
\item\label{it:Ppsi-2} the restriction of $\eta_\psi$ to $\Fc_\psi\subseteq\REu_\psi$ is an injective, normal $*$-homomorphism,
\item\label{it:Ppsi-3} upon identifying $\REu_\psi$ and $\Fc_\psi$ with their images under $\eta_\psi$,
the conditional expectation $F_\psi:\REu_\psi\to\Fc_\psi$ extends to a normal conditional expectation $H_\psi:\Pc_\psi\to\Fc_\psi$,
whose GNS representation is faithful on $\Pc_\psi$,
\item\label{it:Ppsi-4} the image of $\Pc_\psi$ under the GNS representation $\pi_{\psih\circ\Fc_\psi}$ of the normal state $\psih\restrict_{\Fc_\psi}\circ H_\psi$
is the von Neumann algebra $\Mcal_\psi$, so $\Mcal_\psi$ may be identified with $R\Pc_\psi$ for a central projection $R$ of $\Pc_\psi$
and the GNS representation $\pi_{\psih\circ\Fc_\psi}$ itself with $\Pc_\psi\ni x\mapsto Rx$.
\end{enumerate}
\end{propsub}

We now show that the above constructions, in the case that restrictions of $\psih$ to $\Tc_\psi$ or $\Fc_\psi$ have faithful GNS representations, yield
normal conditional expectations from $\Mcal_\psi$ itself,
and we give sufficient conditions for $\Fc_\psi=\Tc_\psi$.
\begin{propsub}\label{prop:N=M}
\begin{enumerate}[(i)]
\item\label{it:N=MbyT} If the restriction of of $\psih$ to $\Tc_\psi$ has faithful GNS representation,
then $\Nc_\psi=\Mcal_\psi$.
Thus, the conditional expectation $F_\psi:\QEu_\psi\to\Tc_\psi$
has an extension to a normal conditional expectation $E_\psi:\Mcal_\psi\to\Tc_\psi$ onto $\Tc_\psi$.
\item\label{it:P=MbyD} If the restriction of of $\psih$ to $\Fc_\psi$ has faithful GNS representation,
then $\Pc_\psi=\Mcal_\psi$.
Thus, the conditional expectation $G_\psi:\REu_\psi\to\Fc_\psi$
has an extension to a normal conditional expectation $H_\psi:\Mcal_\psi\to\Fc_\psi$ onto $\Fc_\psi$.
\item\label{it:T=D} If the restrictions of of $\psih$ to $\Fc_\psi$ and to $\Tc_\psi$ both have faithful GNS representations,
(for example, if the restriction of $\psih$ to $\Fc_\psi$ is faithful), then $\Tc_\psi=\Fc_\psi$.
\end{enumerate}
\end{propsub}
\begin{proof}
If the restriction $\psih\restrict_{\Tc_\psi}$
of $\psih$ to $\Tc_\psi$ has faithful GNS representation, then in the construction found in the proof of Proposition~\ref{prop:Npsi},
we may take $\kappa$ to be the GNS representation of $\psih\restrict_{\Tc_\psi}$.
Thus, $\Vc=L^2(\Tc_\psi,\psih\restrict_{\Tc_\psi})$.
Using~\eqref{eq:W}, we find $\Wc=L^2(\QEu_\psi,\psih\restrict_{\QEu_\psi})$, and, thus, $\Wc$ is naturally identified with $L^2(\Afr,\psi)$,
and the $*$-homomorphism $\sigma_\psi$ is identified with the inclusion of $\QEu_\psi$ into $B(L^2(\Afr,\psi))$ as $\QEu_\psi$ was defined.
Hence, $\Nc_\psi=\Mcal_\psi$.
This proves~\eqref{it:N=MbyT}.

If the restriction of $\psih$ to $\Fc_\psi$ has faithful GNS representation, then by the same argument as above but applied in
the case of the construction described in Proposition~\ref{prop:Ppsi} proves~\eqref{it:P=MbyD}.

If the restrictions of of $\psih$ to $\Fc_\psi$ and to $\Tc_\psi$ both have faithful GNS representations,
then we may apply~\eqref{it:N=MbyT} and~\eqref{it:P=MbyD} to obtain
normal conditional expectations $E_\psi$ and $H_\psi$ as described.
However, these must agree when restriced to $\QEu_\psi$.
Since $\QEu_\psi$ is weakly dense in $\Mcal_\psi$, we have, in fact, $E_\psi=H_\psi$.
Therefore, $\Tc_\psi=\Fc_\psi$.
This proves~\eqref{it:T=D}.
\end{proof}

\begin{examplesub}
In the case of Liu's Example~\ref{ex:Liu}, we have that $\QEu_\psi=\Kt=K(\HEu)+\Cpx1$ is the unitization of the compact operators on
a Hilbert space $\HEu$ with orthonormal basis $(v_n)_{n\ge0}$ and,
letting $(e_{ij})_{i,j\ge0}$ be a corresponding system of matrix units for the compact operators,
we have $\Tc_\psi=\Cpx e_{00}+\Cpx(1-e_{00})$
and $F_\psi(x)=e_{00}xe_{00}+\tau(x)(1-e_{00})$, where $\tau$ is the nonzero character of $\KEut$.
We find $\Wc=\HEu\oplus\Cpx$ and
\[
\sigma_\psi(x)(\xi\oplus z)=x\xi\oplus\tau(x)z\qquad(x\in\QEu_\psi,\,\xi\in\HEu,\,z\in\Cpx),
\]
where $x\xi$ denotes the action of $\QEu_\psi$ on the Hilbert space $\HEu$.
Thus,
\[
B(\HEu)=\Mcal_\psi\subsetneq\Nc_\psi=B(\HEu)\oplus\Cpx.
\]
The projection $R$ in part~\eqref{it:Npsi-4} of Proposition~\ref{prop:Npsi} is $1\oplus0$.
\end{examplesub}

\subsection{The tail $C^*$-algebra}
\label{subsec:tailC*}

Again, we let $\psi$ be a symmetric state on $\Afr$ and we refer to the items at the start of Section~\ref{subsec:normal}
for a summary of the constructions related to the tail algebra $\Tc_\psi$.

\begin{defisub}\label{def:tailC*}
The {\em tail $C^*$-algebra} of $\psi$ is the smallest unital $C^*$-subalgebra $D_\psi$ of $\Tc_\psi$, with the property 
that for every $n\in\Nats$, $x_1,\ldots,x_{n-1}\in D_\psi$, $a_1,\ldots,a_n\in A$ and $i_1,\ldots,i_n\in\Nats$,
we have
\begin{equation}\label{eq:FD}
F_\psi\big(\pi_{i_1}(a_1)x_1\cdots\pi_{i_{n-1}}(a_{n-1})x_{n-1}\pi_{i_n}(a_n)\big)\in D_\psi.
\end{equation}
\end{defisub}

\begin{obssub}\label{obs:condexp}
The restriction of $F_\psi$ to the $C^*$-algebra generated by $\pi_\psi(\Afr)\cup D_\psi$ is a conditional expectation onto $D_\psi$,
and $D_\psi$ is the smallest unital $C^*$-subalgebra of $\Tc_\psi$ having this property.
\end{obssub}

\begin{obssub}\label{obs:Drecurse}
Let $D_{\psi,0}=\Cpx1$ and for $n\ge1$ let $D_{\psi,n}$ be the $C^*$-algebra generated by all expressions of the form appearing
on the left-hand-side of~\eqref{eq:FD} where $x_1,\ldots,x_{n-1}\in D_{\psi,n-1}$.
Then $D_{\psi,n}\subseteq D_{\psi,n+1}$
and $D_\psi$ is the closure of the union $\bigcup_{n=1}^\infty D_{\psi,n}$.
Thus, if $A$ is separable, then so is $D_\psi$.
\end{obssub}

The following question is all the more interesting in light of Proposition~\ref{prop:extSS}, below.
It is answered affirmatively for quantum symmetric states in Proposition~\ref{prop:QSSfreeprod}.
\begin{quessub}\label{ques:D}
Do we have $D_\psi\subseteq\pi_\psi(\Afr)$ for every symmetric state $\psi$?
\end{quessub}

The following question is answered affirmatively for quantum symmetric states in Theorem~\ref{thm:descr}.
\begin{quessub}
Does $D_\psi$ generate $\Tc_\psi$ as a von Neumann algebra?
\end{quessub}

\subsection{Extreme  symmetric states}
\label{subsec:ex}

\begin{defisub}
An {\em extreme symmetric state} of $\Afr=*_1^\infty A$
is an extreme point of the closed convex set $\SSt(A)$ of symmetric states on $\Afr$.
\end{defisub}

The following lemma is a collection of observations
about a convex combination of symmetric states,
arising from the definitions found in Sections~\ref{subsec:tail} and~\ref{subsec:tailC*}.

\begin{lemmasub}\label{lem:statesum}
Suppose $\psi_1$ and $\psi_2$ be symmetric states on $A$, let $0<t<1$ and let $\psi=t\psi_1+(1-t)\psi_2$.
Then there is an inner-product-preserving embedding
\begin{equation}\label{eq:embL2}
L^2(\Afr,\psi)\hookrightarrow L^2(\Afr,\psi_1)\oplus L^2(\Afr,\psi_2),
\end{equation}
where the right-hand-side is equipped with the norm $\|\eta_1\oplus\eta_2\|^2=\|\eta_1\|^2+\|\eta_2\|^2$,
given by
\[
\ah\mapsto t^{1/2}\ah\oplus(1-t)^{1/2}\ah,
\]
which yields an embedding
\begin{equation}\label{eq:embpiA}
\pi_\psi(\Afr)\hookrightarrow\pi_{\psi_1}(\Afr)\oplus\pi_{\psi_2}(\Afr)
\end{equation}
given by
\begin{equation}\label{eq:embpi}
\pi_\psi(x)\mapsto\pi_{\psi_1}(x)\oplus\pi_{\psi_2}(x).
\end{equation}
This, in turn yields embeddings
\begin{align}
\Mcal_\psi&\hookrightarrow\Mcal_{\psi_1}\oplus\Mcal_{\psi_2} \label{eq:embM} \\
\Fc_\psi&\hookrightarrow\Fc_{\psi_1}\oplus\Fc_{\psi_2} \label{eq:embFc} \\
\REu_\psi&\hookrightarrow\REu_{\psi_1}\oplus\REu_{\psi_2} \label{eq:embR} \displaybreak[1] \\
\Tc_\psi&\hookrightarrow\Tc_{\psi_1}\oplus\Tc_{\psi_2} \label{eq:embT} \\
\QEu_\psi&\hookrightarrow\QEu_{\psi_1}\oplus\QEu_{\psi_2} \label{eq:embQ} \\
D_\psi&\hookrightarrow D_{\psi_1}\oplus D_{\psi_2} \label{eq:embD}
\end{align}
and the identities
\begin{align}
G_\psi&=(G_{\psi_1}\oplus G_{\psi_2})\restrict_{\REu_\psi} \label{eq:embG} \\
F_\psi&=(F_{\psi_1}\oplus F_{\psi_2})\restrict_{\QEu_\psi}. \label{eq:embF}
\end{align}
Letting $\chi_1$ and $\chi_2$ be the vector states on $\Mcal_\psi$ corresponding to the vectors $\oneh\oplus0$ and $0\oplus\oneh$,
respectivly, in $L^2(\Afr,\psi_1)\oplus L^2(\Afr,\psi_2)$, we have, for $j=1,2$,
\begin{gather}
\chi_j\restrict_{\QEu_\psi}=\chi_j\circ F_\psi, \label{eq:chirestrict} \\
\chi_j\circ\pi_\psi=\psi_j \label{eq:chipsi} \\
\psih=t\chi_1+(1-t)\chi_2. \label{eq:psihchi}
\end{gather}
\end{lemmasub}
\begin{proof}
The exposition in the lemma pretty much spells out how it goes.
Only some minor comments are in order.
The embeddings~\eqref{eq:embL2}, \eqref{eq:embpiA} and~\eqref{eq:embM} are clear.
The unitary $U_\sigma$ on the Hilbert space $L^2(\Afr,\psi)$ corresponding to $\sigma\in S_\infty$
is the restriction to the embedded copy as in~\eqref{eq:embL2} of the direct sum $U_\sigma^{(1)}\oplus U_\sigma^{(2)}$ of the anaolgous unitaries
on $L^2(\Afr,\psi_1)$ and $L^2(\Afr,\psi_2)$.
Thus, the automorphism $\beta_\sigma=\Ad U_\sigma$ of $\Mcal_\psi$
is the restriction of the direct sum
\begin{equation}\label{eq:embbeta}
\beta_\sigma=(\beta_\sigma^{(1)}\oplus\beta_\sigma^{(2)})\restrict_{\Mcal_\psi}
\end{equation}
of the corresponding automorphisms of $\Mcal_{\psi_1}$ and $\Mcal_{\psi_2}$.
From this we see~\eqref{eq:embFc} and~\eqref{eq:embR}.

The embedding~\eqref{eq:embT} follows by noting
\[
\Tc_\psi=\bigcap_{N\ge1}W^*(\bigcup_{j\ge N}\pi_\psi(\lambda_j(A)))
\subseteq\bigcap_{N\ge1}W^*(\bigcup_{j\ge N}\pi_{\psi_1}(\lambda_j(A))\oplus\pi_{\psi_2}(\lambda_j(A)))=\Tc_{\psi_1}\oplus\Tc_{\psi_2}.
\]
Now~\eqref{eq:embQ} follows from~\eqref{eq:embT} and~\eqref{eq:embpi}.

To prove~\eqref{eq:embG},
note that~\eqref{eq:embbeta} implies that the shift endomorphism $\alpha$ of $\REu_\psi$ is the restriction of the direct sum $\alpha^{(1)}\oplus\alpha^{(2)}$
of corresponding shifts.
Taking the weak limit of powers of this shift yields~\eqref{eq:embG}, and~\eqref{eq:embF} follows by restriction.

The vector state for $\oneh\oplus0$ when restricted to $\Mcal_{\psi_1}\oplus 0$, yields the state $\psih_1$ which, by virtue of Theorem~\ref{thm:E},
satisfies $\psih_1\restrict_{\QEu_{\psi_1}}=\psih_1\circ F_{\psi_1}$, and similarly for $0\oplus\oneh$, $\psih_2$ and $F_{\psi_2}$.
From these facts,~\eqref{eq:chirestrict} and~\eqref{eq:chipsi} follow.
Since $\psih$ is the vector state for the vector $\oneh\in L^2(\Afr,\psi)$,~\eqref{eq:psihchi} follows as well.
\end{proof}

\begin{propsub}\label{prop:extSS}
Let $\psi\in\SSt(A)$.
Let $\rho$ denote the restriction of $\psih$ to the tail $C^*$-algebra $D_\psi$.
Then
\begin{equation}\label{eq:psirho}
\psi=\rho\circ F_\psi\circ\pi_\psi.
\end{equation}
Moreover, 
\begin{enumerate}[(i)]
\item\label{it:pure=>extrSS} if $\rho$ is a pure state of $D_\psi$, then $\psi$ is an extreme point of $\SSt(A)$,
\item\label{it:extrSS=>pure} if $\psi$ is an extreme point of $\SSt(A)$ and if $D_\psi\subseteq\pi_\psi(\Afr)$, then $\rho$ is a pure state of $D_\psi$.
\end{enumerate}
\end{propsub}
\begin{proof}
The identity~\eqref{eq:psirho} is apparent from the fact (Theorem~\ref{thm:E}\eqref{thm:E-1}) that $F_\psi$ is $\psih$-preserving.

Suppose $\psi$ is not an extreme point of $\SSt(A)$.
We may write $\psi=\frac12(\psi_1+\psi_2)$ for $\psi_1,\psi_2\in\SSt(A)$, $\psi_1\ne\psi_2$.
Let $\chi_1$ and $\chi_2$ be the states on $\Mcal_\psi$ defined in Lemma~\ref{lem:statesum} (with $t=\frac12$),
and let $\rho_j$ be the restriction of $\chi_j$ to $D_\psi$.
By~\eqref{eq:psihchi}, we have $\rho=\frac12(\rho_1+\rho_2)$.
By~\eqref{eq:chipsi}, we must have $\chi_1\ne\chi_2$.
But by~\eqref{eq:chirestrict}, this implies $\rho_1\ne\rho_2$.
Thus, $\rho$ is not a pure state of $D_\psi$, and~\eqref{it:pure=>extrSS} is proved.

For~\eqref{it:extrSS=>pure},
suppose $\rho$ is not a pure state.
Then $\rho=\frac12(\rho_1+\rho_2)$ for states $\rho_1$ and $\rho_2$ of $D_\psi$ that are not equal to each other.
Letting $\psi_j=\rho_j\circ F_\psi\circ\pi_\psi$, from~\eqref{eq:psirho}, we have $\psi=\frac12(\psi_1+\psi_2)$.
By the exchangeability for $F_\psi$ (Corollary~\ref{cor:Fexch}), we have $\psi_j\in\SSt(A)$.
Since $D_\psi\subseteq\pi_\psi(\Afr)$, $\rho_j$ can be recovered from $\psi_j$ by restriction, and we must have $\psi_1\ne\psi_2$.
Thus, $\psi$ is not an extreme point of $\SSt(A)$.
\end{proof}

\section{Freeness with amalgamation over the tail algebra}
\label{sec:freeness}

In this section, we show that given a quantum symmetric state, we have freeness
with respect to the conditional expectations $F_\psi$ and $E_\psi$, constructed in the last section,
i.e., freeness with amalgamation over the tail algebra.

We assume $\psi$ is a quantum symmetric state on $\Afr$ and we
continue using the notation of the previous section.
Let
\begin{gather*}
\Bc_i=C^*(\pi_i(A)\cup\Tc_\psi)\subseteq\Mcal_\psi, \\
\Cc_i=W^*(\sigma_\psi(\Bc_i))\subseteq\Nc_\psi.
\end{gather*}

\begin{thm}\label{thm:free}
If $\psi$ is a quantum symmetric state on $\Afr$, then $(\Bc_i)_{i=1}^\infty$ is free
with respect to $F_\psi$ and $(\Cc_i)_{i=1}^\infty$ is free with respect to $E_\psi$.
\end{thm}

The proof of this theorem follows in part Sections~4 and~5 of~\cite{KSp09},
and requires several intermediate results, which are given below.
Again, we let $\pi_i=\pi_\psi\circ\lambda_i$.

\begin{lemma}\label{techlemma}
Assume that $b_{1} , b_{2} , \ldots, b_{n+1} \in \mathcal{T}_{\psi}$.
Then, for $1 \leq i(1) , \ldots , i(n) \leq k$ and $a_{1}, a_{2} , \ldots , a_{n} \in A$,
we have
\begin{multline}\label{eq:bpi}
\psih( b_{1} \pi_{i(1)}(a_{1})b_{2}  \cdots \pi_{i(n)}(a_{n})b_{n+1}) \\
=\sum_{j(1) , \ldots, j(n) = 1}^{k} u_{i(1),j(1)} \cdots u_{i(n),j(n)} \cdot \psih(b_{1} \pi_{j(1)}(a_{1})b_{2}  \cdots \pi_{j(n)}(a_{n})b_{n+1}).
 \end{multline}
 \end{lemma}
\begin{proof}
By Kaplansky's density theorem, each $b_i$ is the limit in strong operator topology of a bounded sequence in the algebra
generated by $\bigcup_{p>k}\pi_p(A)$.
Thus, it will suffice to prove~\eqref{eq:bpi} in the case where each $b_i$ is of the form
\[
b_{i} = \pi_{p(i,1)}(a_{i,1}) \cdots \pi_{p(i,m(i))}(a_{i,m(i)}),
\]
with $m(i)\in\Nats$, $a_{i,j}\in A$ and $k+1 \leq p(i,j) \leq k + \ell$ for all $j = 1,\ldots , m(i)$, for some $\ell>0$.
Showing~\eqref{eq:bpi} in this case amount to showing
\begin{multline}\label{eq:psic}
\psi(c_{1} \lambda_{i(1)}(a_{1})c_{2}  \cdots \lambda_{i(n)}(a_{n})c_{n+1}) \\
=\sum_{j(1),\ldots,j(n) = 1}^{k} u_{i(1),j(1)} \cdots u_{i(n),j(n)}
   \cdot \psi(c_{1} \lambda_{j(1)}(a_{1})c_{2}  \cdots \lambda_{j(n)}(a_{n})c_{n+1})
\end{multline}
where
\begin{equation}\label{eq:ci}
c_i = \lambda_{p(i,1)}(a_{i,1}) \cdots \lambda_{p(i,m(i))}(a_{i,m(i)}).
\end{equation}

Consider the matrix $U_{k} = (u_{ij})_{i,j = 1}^{k}$ for the defining generators $u_{ij}$ of the quantum permutation group $A_{s}(k)$.
Consider the $(k+\ell)\times(k+\ell)$
matrix $\tilde{U}_{k+\ell} := U_{k} \oplus 1_{\ell} = (\tilde{u}_{ij})_{i,j = 1}^{k + \ell}$ where $\tilde{u}_{ij} = u_{ij}$ for $i,j \leq k$ and
$\tilde{u}_{ij} = \delta_{ij}1$ when either $i > k$ or $j > k$.  The entries of $\tilde{U}_{k + \ell}$ satisfy
the defining relations of $A_{s}(k + \ell)$ and yield a corresponding $*$-homomorphism from $A_s(k+\ell)$ onto $A_s(k)$.
We now use the quantum exchangeability of the representations $\lambda_1,\lambda_2,\ldots$ with respect to $\psi$ and compose
with the $*$-homomorphism $A_s(k+\ell)\to A_2(k)$ described above to obtain that
the list $\lambdat_1,\ldots,\lambdat_{k+\ell}$ has the same distribution with respect to $\id\otimes\psi$ on $A_s(k)\otimes\Afr$
as does $\lambda_1,\ldots,\lambda_{k+\ell}$ with respect to $\psi$ on $\Afr$, where
\[
\lambdat_i=\sum_{j=1}^{k+\ell}\ut_{i,j}\otimes\lambda_j:A\to A_s(k)\otimes\Bfr.
\]
But $\ut_{i,j}=\delta_{i,j}1$ if either $i>k$ or $j>k$, so
\[
\lambdat_i=\begin{cases}
1\otimes\lambda_i,&i>k \\
\sum_{j=1}^k u_{i,j}\otimes\lambda_j,&i\le k.
\end{cases}
\]
Thus, when, in the expression~\eqref{eq:ci} for $c_i$, the $\lambda_{p(i,j)}$ are replaced by $\lambdat_{p(i,j)}$, we get $1\otimes c_i$.
So applying the aforementioned equidistribution to the expression
\[
c_{1} \lambda_{i(1)}(a_{1})c_{2}  \cdots \lambda_{i(n)}(a_{n})c_{n+1},
\]
we get~\eqref{eq:psic}.
\end{proof}

We now fix notation for the remainder of the section.  Let $a_{1} , \ldots , a_{N} \in A$ and consider the algebra
$\mathcal{T}_{\psi} \langle X_{1} , \ldots , X_{N} \rangle$ of noncommutative polynomials, in noncommuting indeterminants
$X_1,\ldots,X_N$ that also do not commute with the coefficient algebra $\Tc_\psi$.
Take $P_{j}(X_{1} , \ldots , X_{N}) \in \mathcal{T}_{\psi} \langle X_{1} , \ldots , X_{N} \rangle$ for $j = 1 ,\ldots , n$,
and for $i \in \mathbb{N}$ define 
\begin{equation}\label{eq:xP}
x_{j}^{i} := P_{j}(\pi_{i}(a_{1}), \ldots , \pi_{i}(a_{N})).
\end{equation}

As an immediate consequence of the exchangeability of $\psih$ (Lemma~\ref{lem:R0exch}),
for any permutation $\sigma\in S_\infty$, we have
\begin{equation}\label{eq:psiexch}
\psih(x_1^{i(1)}\cdots x_n^{i(n)})=\psih(x_1^{\sigma(i(1))}\cdots x_n^{\sigma(i(n))}).
\end{equation}
The next result shows quantum exchangeability.
\begin{lemma}\label{poly_corr}
For natural numbers  $1 \leq i(1), \ldots , i(n) \leq k$,
we have  
\begin{equation}\label{eq:ncpolys}
\psih(x_{1}^{i(1)} \cdots x_{n}^{i(n)}) =  \sum_{j(1), \ldots , j(n) = 1}^{k} u_{i(1),j(1)} \cdots u_{i(n),j(n)} \psih(x_{1}^{j(1)} \cdots x_{n}^{j(n)}).
\end{equation}
\end{lemma}
\begin{proof}
It will suffice to show it when each $P_i$ is a monomial, i.e., is of the form
$d_1X_{f(1)}\cdots d_pX_{f(p)}d_{p+1}$ for some $p\in\Nats$ and $d_1,\ldots,d_{p+1}\in\Tc_\psi$.
Now this case of~\eqref{eq:ncpolys} will follow if we show that, in the setting of Lemma~\ref{techlemma},
whenever we have constant blocks $i(\ell_q+1)=i(\ell_q+2)=\cdots=i(\ell_q+m_q)$ for all $1\le q\le Q$,
with $m_q\ge1$, $\ell_1=0$, $\ell_q=m_1+\cdots+m_{q-1}$ and $m_1+\cdots+m_Q=n$, then~\eqref{eq:bpi} becomes
\begin{multline}\label{eq:jsum}
\psih( b_{1} \pi_{i(1)}(a_{1})b_{2}  \cdots \pi_{i(n)}(a_{n})b_{n+1}) \\
=\sum_{\substack{1\le j(1),\ldots,j(n)\le k,\\j(\ell_q+1)=\cdots=j(\ell_q+m_q),\;(1\le q\le Q)}}
 u_{i(1),j(1)} \cdots u_{i(n),j(n)} \cdot \psih(b_{1} \pi_{j(1)}(a_{1})b_{2}  \cdots \pi_{j(n)}(a_{n})b_{n+1}).
 \end{multline}
But this follows easily, because the terms in the summation on the right hand side of~\eqref{eq:bpi} that don't appear in
the summation of~\eqref{eq:jsum} are all zero.
Indeed, this follows directly from the defining relations of $A_s(k)$, which imply $u_{i,j}u_{i,j'}=\delta_{j,j'}u_{i,j}$.
\end{proof}

By Corollary~\ref{cor:Fexch}, we have exchangeability with respect to $F_\psi$:
for any permutation $\sigma \in S_{\infty}$,
\begin{equation}\label{eq:Fxsymm}
F_\psi[x_{1}^{i(1)} \cdots  x_{n}^{i(n)}]= F_\psi[x_{1}^{\sigma(i(1))} \cdots  x_{n}^{\sigma(i(n))}].
\end{equation}
The next result shows that we have, in a sense, also quantum exchangeability with respect to the expectation $F_\psi$.
\begin{prop}\label{expProp}
For $1 \leq i(1) , \ldots , i(n) \leq k$, we have
\begin{equation}\label{eq:1E}
1 \otimes F_\psi[x_{1}^{i(1)} \cdots  x_{n}^{i(n)}] = \sum_{j(1) , \ldots, j(n) = 1}^{k} u_{i(1),j(1)} \cdots u_{i(n),j(n)} \otimes F_\psi[x_{1}^{j(1)} \cdots  x_{n}^{j(n)}]
\end{equation}
\end{prop}
\begin{proof}
It will suffice to show that for every $y,z\in\Afr$ of the form
\begin{align}
y&=\lambda_{p(1)}(a_1')\cdots\lambda_{p(s)}(a_s') \label{eq:ya'} \\
z&=\lambda_{q(1)}(a_1'')\cdots\lambda_{q(t)}(a_t'') \label{eq:za''}
\end{align}
for positive integers $s,t,p(j),q(j)$ and for $a_j',a_j''\in A$,
we have
\begin{multline}\label{eq:psixEy}
1\cdot\psi(z^*F_\psi(x_1^{i(1)}\cdots x_n^{i(n)})y) \\
=\sum_{j(1),\ldots,j(n)=1}^ku_{i(1),j(1)}\cdots u_{i(n),j(n)}\psi(z^*F_\psi(x_1^{j(1)}\cdots x_n^{j(n)})y),
\end{multline}
since $\Tc_\psi$ acts faithfully on $L^2(\Afr,\psi)$, and since the linear span of the vectors $\yh$ for $y$ as in~\eqref{eq:ya'}
is dense in this Hilbert space.

Let $R=\max(p(1),\ldots,p(s),q(1),\ldots,q(t))$.
By Remark~\ref{rem:WOTconst},
\begin{multline}\label{eq:useEave}
\psih(z^*F_\psi(x_1^{i(1)}\cdots x_n^{i(n)})y) \\
=\psih\big(\pi_{q(t)}(a_t'')^*\cdots\pi_{q(1)}(a_1'')^*
 x_1^{i(1)+R}\cdots x_n^{i(n)+R}\pi_{p(1)}(a_1')\cdots\pi_{p(s)}(a_s')\big).
\end{multline}
Using~\eqref{eq:psiexch} and an appropriate permutation $\sigma$, we find
\begin{multline}\label{eq:aftersigma}
\psih(\pi_{q(t)}(a_t'')^*\cdots\pi_{q(1)}(a_1'')^*
 x_1^{i(1)+R}\cdots x_n^{i(n)+R}
 \pi_{p(1)}(a_1')\cdots\pi_{p(s)}(a_s')) \\
=\psih(\pi_{q(t)+k}(a_t'')^*\cdots\pi_{q(1)+k}(a_1'')^*
 x_1^{i(1)}\cdots x_n^{i(n)}
 \pi_{p(1)+k}(a_1')\cdots\pi_{p(s)+k}(a_s')).
\end{multline}
Now, using (a) Lemma~\ref{techlemma},
and the same quotient map $A_s(k+R)\to A_s(k)$ as appeared in the proof of Lemma~\ref{techlemma}
(but with $R$ instead of $\ell$), we find
\begin{multline*}
\psih(\pi_{q(t)+k}(a_t'')^*\cdots\pi_{q(1)+k}(a_1'')^*
 x_1^{i(1)}\cdots x_n^{i(n)}
 \pi_{p(1)+k}(a_1')\cdots\pi_{p(s)+k}(a_s')) \\
=\sum_{j(1),\ldots,j(n)=1}^k
u_{i(1),j(1)}\cdots u_{i(n),j(n)}
\;\psih\big(\pi_{q(t)+k}(a_t'')^*\cdots\pi_{q(1)+k}(a_1'')^*
 x_1^{j(1)}\cdots x_n^{j(n)} \\
 \;\pi_{p(1)+k}(a_1')\cdots\pi_{p(s)+k}(a_s')\big),
 \end{multline*}
while applying the argument used in~\eqref{eq:useEave} and~\eqref{eq:aftersigma} in reverse, we get
\begin{multline*}
\psih\big(\pi_{q(t)+k}(a_t'')^*\cdots\pi_{q(1)+k}(a_1'')^*
 x_1^{j(1)}\cdots x_n^{j(n)}\pi_{p(1)+k}(a_1')\cdots\pi_{p(s)+k}(a_s')\big)= \\
=\psih(z^*F_\psi(x_1^{j(1)}\cdots x_n^{j(n)})y).
\end{multline*}
Thus, we have~\eqref{eq:psixEy}, as required.
\end{proof}

In what follows, we will abuse notation and rewrite~\eqref{eq:1E} as
$$ F_\psi[x_{1}^{i(1)} \cdots  x_{n}^{i(n)}] = \sum_{j(1) , \ldots, j(n) = 1}^{k} u_{i(1),j(1)} \cdots u_{i(n),j(n)} \cdot F_\psi[x_{1}^{j(1)} \cdots  x_{n}^{j(n)}]$$
where the tensor product structure is implicit.

\begin{prop}\label{ExpProp}
Let $1 \leq i(1) , \ldots , i(n) \leq k$ and assume that for fixed $\ell$ we have that $i(\ell) \neq i(j)$ for all $j \neq \ell$.
Then we have
$$F_\psi[x_{1}^{i(1)} \cdots  x_{n}^{i(n)}] =  F_\psi[x_{1}^{i(1)} \cdots F_\psi[x_{\ell}^{i(\ell)}] \cdots  x_{n}^{i(n)}] .$$
\end{prop}
\begin{proof}
It will suffice to show, for arbitrary $y$ and $z$ as in~\eqref{eq:ya'} and~\eqref{eq:za''}, that we have
\begin{equation}\label{eq:zEy}
\psih(z^*F_\psi[x_{1}^{i(1)} \cdots  x_{n}^{i(n)}]y)=\psih(z^*F_\psi[x_{1}^{i(1)} \cdots F_\psi[x_{\ell}^{i(\ell)}] \cdots  x_{n}^{i(n)}]y).
\end{equation}
By Remark~\ref{rem:WOTconst}, the left hand side of~\eqref{eq:zEy} equals
\begin{equation}\label{eq:zEyLHS}
\psih(z^*x_1^{i(1)+N}\cdots x_n^{i(n)+N}y)
\end{equation}
while the right hand side of~\eqref{eq:zEy} equals
\begin{equation}\label{eq:lout}
\psih(z^*x_1^{i(1)+N}\cdots x_{\ell-1}^{i(\ell-1)+N}F_\psi(x_\ell^{i(\ell)})x_{\ell+1}^{i(\ell+1)+N}\cdots x_n^{i(n)+N}y).
\end{equation}
for all $N$ sufficiently large.
Applying Remark~\ref{rem:WOTconst} again, the quantity~\eqref{eq:lout} is equal to
\begin{equation}\label{eq:NM}
\psih(z^*x_1^{i(1)+N}\cdots x_{\ell-1}^{i(\ell-1)+N}x_\ell^{i(\ell)+M}x_{\ell+1}^{i(\ell+1)+N}\cdots x_n^{i(n)+N}y).
\end{equation}
for all $M$ sufficiently large.
Choosing $N$ and then $M$ large enough, choosing an appropriate permutation and using the exchangeability result Lemma~\ref{lem:R0exch},
we have that the quantities~\eqref{eq:zEyLHS} and~\eqref{eq:NM} agree.
\end{proof}

The following lemma was essential in the proof of the main theorem in Section 5 of ~\cite{KSp09} 
and will be used in a similar way to prove Theorem \ref{thm:free}.  We refer to that paper for proof of this lemma.
\begin{lemma}\label{SpeichLemma}
Consider self-adjoint projections $p$ and $q$.
Assume that $s \geq 2$.  The following are equivalent:
\begin{enumerate}[(i)]
 \item $(pq)^{s} + (p(1-q))^{s} + ((1-p)q)^{s} + ((1-p)(1-q))^{s} = 1$,
\item  $(pq)^{s}p + (p(1-q))^{s}p + ((1-p)q)^{s}(1-p) + ((1-p)(1-q))^{s}(1-p) = 1$,
\item $p$ and $q$ commute.
\end{enumerate}
\end{lemma}

We now have all of the pieces in place to prove Theorem \ref{thm:free}.

\begin{proof}[Proof of Theorem~\ref{thm:free}.]
It will suffice to show freeness of $(\Bc_i)_{i=1}^\infty$ with respect to $F_\psi$.
Indeed, since $E_\psi$ is a normal extension of $F_\psi$, freeness of $(\Cc_i)_{i=1}^\infty$ with respect to $E_\psi$
will follow.

Fix $n \in \mathbb{N}$ and natural numbers $i(1) \neq i(2), \ldots , i(n-1) \neq i(n)$.
Assume $F_\psi[x_{j}^{k}] = 0$ for $j = 1, \ldots , n$.
We will show $F_\psi[x_{1}^{i(1)} \cdots x_{n}^{i(n)}] = 0$, thereby proving freeness over the tail algebra.

Toward this end, putting $\textbf{i} := (i(1) , \ldots , i(n))$, we denote by $\ker\textbf{i}$ the partition of the set $\{ 1, 2 , \ldots , n \}$ whereby
$j$ and $\ell$ belong to the same block if and only if $i(j) = i(\ell)$.
If $\ker\textbf{i}$ were non-crossing, then one of its blocks would have to be an interval block;
but by assumption, $i(p)\ne i(p+1)$ for all $p$, so this interval would have to be a singleton.
If $\ker\textbf{i}$ has a block consisting of only one element, then by Proposition~\ref{ExpProp}
we have $F_\psi[x_{1}^{i(1)} \cdots x_{n}^{i(n)}] =0$, as desired.

We proceed by induction on the number $r$ of blocks in $\ker\textbf{i}$, starting with $r=n$ and decreasing from there.
If $r=n$ then $\ker\textbf{i}$ must have a singleton block and this case is done.
Now we introduce the induction hypothesis for $\ker\textbf{i}$ with at least $r+1$ blocks and use this to address the case where
$\ker\textbf{i}$ has $r$ blocks.

Utilizing Proposition \ref{expProp}, we have
\begin{align}
 F_\psi[x_{1}^{i(1)} \cdots x_{n}^{i(n)}]
&= \sum_{j(1) , \ldots, j(n) = 1}^{k} u_{i(1),j(1)} \cdots u_{i(n),j(n)} \cdot F_\psi[x_{1}^{j(1)} \cdots x_{n}^{j(n)}] \notag \\[1ex]
&= \sum_{\pi \in \mathcal{P}(n)} \sum_{ \substack{\ker(\textbf{j}) =  \pi \\ 1 \leq j(1), \ldots , j(n) \leq k}}  u_{i(1),j(1)} \cdots u_{i(n),j(n)} \cdot F_\psi[x_{1}^{j(1)} \cdots x_{n}^{j(n)}]. \label{eq:uijs}
\end{align}
Observe that if $j(\ell) = j(\ell + 1)$ then $u_{i(\ell),j(\ell)}u_{i(\ell + 1),j(\ell + 1)} = 0$,
since $i(\ell) \neq i(\ell + 1)$.
Thus, we may restrict to the cases where
$j(1) \neq j(2) , \ldots , j(n-1) \neq j(n)$.
If the number of blocks is at least $r+1$,
then $F_\psi[x_{1}^{j(1)} \cdots x_{n}^{j(n)}]=0$ by the induction hypothesis,
so we may restrict to the cases where $\pi$ has at most $r$ blocks.

We will chose a particular representation of $A_s(n)$ and apply it to~\eqref{eq:uijs}.
It will be convenient to assume that $i$ takes values only in the set
$\{ 1,3, 5, \ldots, 2r - 1 \}$,
which we can do by the exchangeability found in equation~\eqref{eq:Fxsymm}.
Now let $\{ q_{m} \}_{m=1}^{r}$ be a family of self-adjoint projections in $M_2(\Cpx)$, to be specified later,
and consider the unitary matrix
$$\ut_{m} =  \left( \begin{array}{cc} q_{m} & 1-q_{m} \\ 
                 1 - q_{m} & q_{m} 
                 \end{array}\right)\in M_4(\Cpx).$$
Consider the $2r \times 2r$ block matrix $\ut = \bigoplus_{m=1}^{r} \ut_{m}\in M_{4r}(\Cpx)$
and let $(\ut_{i,j})_{1\le i,j\le 2r}$ be the corresponding entries.
Thus, we have
\[
\ut_{i,j}=\begin{cases}
q_m,&i=j\in\{2m-1,2m\} \\
1-q_m,&\{i,j\}=\{2m-1,2m\} \\
0,&\text{otherwise.}
\end{cases}
\]
It is easy to check that the $\ut_{i,j}$ satisfy the defining relations for $A_{s}(2r)$, so there is a $*$-representation from
$A_s(2r)$ sending $u_{i,j}$ to $\ut_{i,j}$.
We apply this $*$-representation to~\eqref{eq:uijs}.

By assumption, the indices for $\textbf{i} = (i(1) , \ldots , i(n))$ take on only odd numbers.
Thus, for non-vanishing $\ut_{i,j}$, the $j$ value determines the $i$
value since there are only 2 non-zero entries in the $j$-th row of $\ut$, only one of which is odd.
Therefore, $\ker\textbf{j} \leq \ker\textbf{i}$.
Thus, in the sum resulting from applying this $*$-representation, we may assume that $\pi \leq \ker\textbf{i}$.
As we have already discarded the cases where $|\pi| > r$ we are left with one choice, namely $\pi = \ker\textbf{i}$.
Under these conditions, we have
\begin{align*}
F_\psi[x_{1}^{i(1)} \cdots x_{n}^{i(n)}]
&= \sum_{\substack{1 \leq j(1), \ldots , j(n) \leq k \\ \ker(\textbf{j}) = \ker(\textbf{i})}}\ut_{i(1),j(1)} \cdots\ut_{i(n),j(n)} \cdot F_\psi[x_{1}^{j(1)} \cdots x_{n}^{j(n)}] \\
&= \bigg( \sum_{\substack{1 \leq j(1), \ldots , j(n) \leq k \\ \ker(\textbf{j}) = \ker(\textbf{i})}}\ut_{i(1),j(1)} \cdots\ut_{i(n),j(n)}  \bigg) \cdot F_\psi[x_{1}^{i(1)} \cdots x_{n}^{i(n)}],
\end{align*}
where the last equality follows easily by the exchangeability found in equation~\eqref{eq:Fxsymm}.

Thus, if the sum
\begin{equation}\label{eq:sumus}
\sum_{\substack{1 \leq j(1), \ldots , j(n) \leq k \\ \ker(\textbf{j}) = \ker(\textbf{i})}}  \ut_{i(1),j(1)} \cdots \ut_{i(n),j(n)}
\end{equation}
is not equal to $1$, then we may conclude $F_\psi[x_{1}^{i(1)} \cdots x_{n}^{i(n)}] = 0$, proving our theorem.

As observed at the start of the proof, we may assume $\ker\textbf{i}$ is crossing, 
and we choose two blocks that maintain this crossing.
The values of $\textbf{i}$ on these two crossing blocks are $2\ell-1$ and $2m-1$ for some distinct $\ell,m\in\{1,\ldots,r\}$.
We now choose our projections $q_i$, taking $q_\ell$ and $q_m$ that do not commute, and letting $q_i=1$ for all other values of $i$.
Now a careful analysis shows that~\eqref{eq:sumus} reduces to one of the following:
\begin{gather*}
(q_{m}q_{\ell})^{s} + (q_{m}(1-q_{\ell}))^{s} + ((1-q_{m})q_{\ell})^{s} + ((1-q_{m})(1-q_{\ell}))^{s} \\[1ex]
(q_{m}q_{\ell})^{s}q_{m} + (q_{m}(1-q_{\ell}))^{s}q_{m} + ((1-q_{m})q_{\ell})^{s}(1-q_{m}) + ((1-q_{m})(1-q_{\ell}))^{s}(1-q_{m}) \\[1ex]
(q_{\ell}q_{m})^{s}q_{\ell} + (q_{\ell}(1-q_{m}))^{s}q_{\ell} + ((1-q_{\ell})q_{m})^{s}(1-q_{\ell}) + ((1-q_{\ell})(1-q_{m}))^{s}(1-q_{\ell})
\end{gather*}
for some $s \geq 2$.
By Lemma \ref{SpeichLemma}, these are all three distinct from $1$, and our theorem is proved.
\end{proof}

The first part of the next result answers Question~\ref{ques:D} affirmatively in the case of a quantum symmetric state.
The second part is a careful phrasing of the fact, now essentially obvious from Theorem~\ref{thm:free},
that $\pi(\Afr)$ is the free product with amalgamation over the tail $C^*$-algebra of the images of $A$.
It gives a sort of converse to Proposition~\ref{prop:freeqsymm}.

\begin{prop}\label{prop:QSSfreeprod}
Let $\psi$ be a quantum symmetric state on $\Afr$.
Then $D_\psi\subseteq\pi_\psi(\Afr)$.
Moreover, there are
\renewcommand{\labelitemi}{$\bullet$}
\begin{itemize}
\item a unital $C^*$-algebra $B$,
\item a conditional expectation $F:B \to D$ onto a unital $C^*$-subalgebra $D$ of $B$ whose GNS representation is faithful (on $B$),
\item a unital $*$-homomorphism $\theta:A\to B$
\item a state $\rho$ of $D$
\end{itemize}
such that
\begin{itemize}
\item $B$ is generated by $\theta(A)\cup D$,
\end{itemize}
and, letting
\begin{equation}\label{eq:*BF}
(Y,G)=(*_D)_{i=1}^\infty(B,F)
\end{equation}
be the $C^*$-algebra free product of infinitely many copies of $(B,F)$ with itself, amalgamated over $D$,
there is an identification
\[
(\pi_\psi(\Afr),F_\psi\restrict_{\pi_\psi(\Afr)})=(Y,G),
\]
where $F_\psi$ is the conditional expectation constructed in
Theorem~\ref{thm:E} (see also Observation~\ref{obs:condexp})
with this free product pair,
in such a way that
\begin{enumerate}[(i)]
\item\label{it:QSSfreeprod-3} the tail $C^*$-algebra $D_\psi$ in $\pi_\psi(\Afr)$ is identified with the copy of $D$
over which the amalgamation is performed in~\eqref{eq:*BF}
\item\label{it:QSSfreeprod-4} the $*$-homomorphism $\pi_\psi$ corresponds to the free product $*_1^\infty\theta$ of infinitely many copies of $\theta$
\item\label{it:QSSfreeprod-5} the state $\psih$ restricted to $\pi_\psi(\Afr)$ corresponds to the state $\rho\circ G$ of $Y$.
\end{enumerate}
\end{prop}
\begin{proof}
Let $B_i=C^*(D_\psi\cup\pi_i(A))\subseteq\Mcal_\psi$.
By Theorem~\ref{thm:free} and Observation~\ref{obs:condexp}, the family $(B_i)_{i=1}^\infty$ is free with respect to $F_\psi$.
Moveover, we have
$C^*(\bigcup_{i=1}^\infty B_i)=C^*(\pi_\psi(\Afr)\cup D_\psi)$.
Denoting this $C^*$-algebra by $Y$, the GNS representation of the restriction of $\psih$ to $Y$ is, of course, faithful.
Since $F_\psi$ is $\psih$-preserving, alse the GNS representation of the restriction of $F_\psi$ to $Y$ is faithful.
Thus, $(Y,F_\psi\restrict_Y)$ is isomorphic to the amalgamated free product of $C^*$-algebras
\begin{equation}\label{eq:Y}
(Y,F_\psi\restrict_Y)=(*_{D_\psi})_{i=1}^\infty(B_i,F_\psi\restrict_{B_i}).
\end{equation}
By the exchangeability found in equation~\eqref{eq:Fxsymm},
each $(B_i,F_\psi\restrict_{B_i})$ isomorphic $(B_1,F_\psi\restrict_{B_1})$
by an isomorphism that is the identity on  $D_\psi$ and intertwines $\lambda_i$ and $\lambda_1$.
Thus, we may write $D$ for $D_\psi$, $(B,F)$ for $(B_1,F_\psi\restrict_{B_1})$ and $\theta:A\to B$ for the $*$-homomorphism
$\pi_1=\pi_\psi\circ\lambda_1$.

Letting $\alpha$ denote the shift $*$-endomorphism of $Y$, that sends $B_i$ to $B_{i+1}$, that we have, thanks to the free product
realization~\eqref{eq:Y} and Theorem~6.1 of~\cite{AD09}, that for every $y\in Y$, the averages
\[
\frac1n\sum_{k=1}^n\alpha^k(y)
\]
converge in norm to $F_\psi(y)$.
But this $*$-endomorphism $\alpha$ agrees with the one constructed on Lemma~\ref{lem:alpha} (when the latter is restricted to $Y$).
Since $\alpha(\pi_\psi(\Afr))\subseteq\pi_\psi(\Afr)$,
the inclusion $F_\psi(\pi_\psi(\Afr))\subseteq\pi_\psi(\Afr)$ follows.
Thus, using the recursive description of the generation of $D_\psi$ from Observation~\ref{obs:Drecurse}, we get $D_\psi\subseteq\pi_\psi(\Afr)$.

Now, letting $\rho$ be the restriction of $\psih$ to $D=D_\psi$,
the assertions of the proposition follow easily.
\end{proof}

\section{Description of quantum symmetric states}
\label{sec:descr}

Now we combine Proposition~\ref{prop:freeqsymm} and Proposition~\ref{prop:QSSfreeprod}
to formulate a description
of quantum symmetric states on $\Afr$ in terms of amalgamated free products of von Neumann algebras.
Part of this result is a more complete version of Proposition~3.1 of~\cite{DK}.

\begin{defi}\label{def:VA}
For a unital $C^*$-algebra $A$, let $\VEu(A)$ be the set of all equivalence classes of quintuples $(B,D,F,\theta,\rho)$, such that
\begin{enumerate}[(i)]
\item $B$ is a $C^*$-algebra
\item $D$ is a unital $C^*$-subalgebra of $B$,
\item $F:B\to D$ is a conditional expectation onto $D$,
\item\label{item:faithfulGNS} the GNS representation of $F$ is a faithful representation of $B$,
\item $\theta:A\to B$ is a unital $*$-homomorphism,
\item\label{item:ADgen} $\theta(A)\cup D$ generates $B$ as a $C^*$-algebra,
\item\label{item:Dsmallest} $D$ is the smallest unital $C^*$-subalgebra of $B$ that satisfies
\begin{equation}\label{eq:Fdas}
F\big(x_0\theta(a_1)x_1\cdots\theta(a_n)x_n\big)\in D
\end{equation}
whenever $n\in\Nats$, $x_0,\ldots,x_n\in D$ and $a_1,\ldots,a_n\in A$,
\item $\rho$ is a state on $D$,
\item\label{item:fpGNSfaithful} letting $(Y,G)=(*_D)_{i=1}^\infty(B,F)$ denote the amalgamated free product of $C^*$-noncommutative probability spaces
and letting $\pi_{\rho\circ G}$ be the GNS representation of $Y$ corresponding to the state $\rho\circ G$ of $Y$,
we have $\ker\pi_{\rho\circ G}\cap D=\{0\}$, i.e., $\pi_{\rho\circ G}$ is faithful when restricted to $D$,
\end{enumerate}
and where quintuples $(B,D,F,\theta,\rho)$ and $(B',D',F',\theta',\rho')$ are defined to be equivalent
if there is a $*$-isomorphism $\pi:B\to B'$ sending $D$ onto $D'$ and so that $\pi\circ F=F'\circ\pi$,
$\pi\circ\theta=\theta'$ and $\rho'\circ\pi\restrict_D=\rho$.
\end{defi}

\begin{remarks}
\begin{enumerate}[(a)]
\item
In order to avoid set theoretic difficulties, instead of speaking of about the set of equivalence classes of all quintuples,
we should note that for a given $A$, conditions~\eqref{item:Dsmallest} and~\eqref{item:ADgen} impose a limit on the
cardinality of a dense subset of $B$, so we may choose a particular Hilbert space $\HEu$ and work with the set of 
all equivalence classes of quintuples where $B$ is a $C^*$-algebra in $B(\HEu)$.
However, in practice we will ignore this issue and use the sloppy language ``all quintuples.''
\item
In practice, we will surpress notation for equivalence classes, and just write $(B,D,F,\theta,\rho)\in\VEu(A)$;
implicitly, we will identify two such quintuples as being the same if there is a $*$-isomorphism $\pi$ as described
in the definition above.
\item The condition~\eqref{item:fpGNSfaithful} is slightly subtle, as it may be satisfied even when the GNS representation $\pi_{\rho\circ F}$ of
the state $\rho\circ F$ of $B$ fails to be faithful when restricted to $D$.
Note that, by hypothesis~\eqref{item:faithfulGNS} and the free product construction,
the GNS representation of $Y$ (on the Hilbert $C^*$-module $L^2(Y,G)$) arising from $G$ is faithful.
Moreover, by Proposition~\ref{prop:rhoFt}, condition~\eqref{item:fpGNSfaithful} is equivalent to faithfulness of $\pi_{\rho\circ G}$ itself.
\end{enumerate}
\end{remarks}

Here is the classification of quantum symmetric states.

\begin{thm}\label{thm:descr}
There is a bijection
\begin{equation}\label{eq:bijVQ}
\VEu(A)\to\QSS(A)
\end{equation}
that assigns to $V=(B,D,F,\theta,\rho)\in\VEu(A)$ the state $\psi=\psi_V$
given as follows:
letting 
\begin{equation}\label{eq:YGfp}
(Y,G)=(*_D)_{i=1}^\infty(B,F)
\end{equation}
be the amalgamated free product and letting $*_1^\infty\theta:\Afr\to Y$ be the free product (arising from the univeral property) of the
homomorphisms from the copies of $A$ into the respective copies of $B$, 
we set $\psi=\rho\circ G\circ(*_1^\infty\theta)$.

We let $\pi_{\rho\circ G}$ denote the GNS representation of the state $\rho\circ G$ on $Y$.
Then we have the correspondence shown in Table~\ref{tab:cor} between objects defined in Section~\ref{sec:tails} and objects associated to $(Y,G)$,
\begin{table}[hbt]
\caption{Correspondence between $(B,D,F,\theta,\rho)\in\VEu(A)$ and $\psi\in\QSS(A)$.}\label{tab:cor}
\begin{tabular}{c|c|r}
from $(Y,G)$ & from Section~\ref{sec:tails} & label \\ \hline\hline
\rule{0ex}{2.5ex}$L^2(Y,\rho\circ G)$ & $L^2(\Afr,\psi)$ & (a) \\ \hline
\rule{0ex}{2.5ex}$Y$& $\pi_\psi(\Afr)$ & (b) \\ \hline
\rule{0ex}{2.5ex}$\pi_{\rho\circ G}(Y)''$ & $\Mcal_\psi$ & (c) \\ \hline
\rule{0ex}{2.5ex}$\pi_{\rho\circ G}(D)''$ & $\Tc_\psi$ & (d) \\ \hline
\rule{0ex}{2.5ex}$G$ & $F_\psi\restrict_{\pi_\psi(\Afr)}$ & (e) \\ \hline
\rule{0ex}{2.5ex}$D$ & $D_\psi$ & (f) \\ \hline
\rule{0ex}{2.5ex}$B_i$ & $C^*(\pi_\psi(\lambda_i(A))\cup D_\psi)$ & (g) \\ \hline
\end{tabular}
\end{table}
where $B_i$ denotes the $i$-th copy of $B$ in $Y$
and where,
(since by condition \ref{def:VA}\eqref{item:fpGNSfaithful} and Proposition~\ref{prop:rhoFt}, $\pi_{\rho\circ G}$ is faithful on $Y$)
in the left column of Table~\ref{tab:cor}, we implicitly identify $Y$ and $C^*$-subalgebras of $Y$ with their corresponding images under $\pi_{\rho\circ G}$.
\end{thm}
\begin{proof}
Given $V=(B,D,F,\theta,\rho)\in\VEu(A)$,
it follows from Proposition~\ref{prop:freeqsymm} that $\psi=\psi_V$ as described is
a quantum symmetric state on $\Afr$.
We will now verify the identifications indicated in Table~\ref{tab:cor}.

By choice of $\psi$, we have
\begin{equation}\label{eq:L2s}
L^2(\Afr,\psi)=L^2(\Afr,\rho\circ G\circ(*_1^\infty\theta)).
\end{equation}
An argument used in the proof of Proposition~\ref{prop:QSSfreeprod} shows that $C^*((*_1^\infty\theta)(\Afr))$ is closed under $G$.
Indeed, it is clearly closed under the ``free right shift'' endomorphism $\alpha$ of $Y$, and by Theorem~6.1 of~\cite{AD09}, $G(y)$
is the norm limit
\begin{equation}\label{eq:Gylim}
G(y)=\lim_{n\to\infty}\frac1n\sum_{k=1}^n\alpha^k(y).
\end{equation}
Now condition~\ref{def:VA}\eqref{item:Dsmallest} yields that $C^*((*_1^\infty\theta)(\Afr))$ contains $D$,
so condition~\ref{def:VA}\eqref{item:ADgen} shows that $C^*((*_1^\infty\theta)(\Afr))$ is all of $Y$.
This implies~(b) and also 
allows us to make the identification $L^2(\Afr,\rho\circ G\circ(*_1^\infty\theta))=L^2(Y,\rho\circ G)$,
which yields~(a).

Taking von Neumann algebra closures yields a normal $*$-isomorphism
$\pi_\psi(\Afr)''\to\pi_{\rho\circ G}(Y)''$,
which is~(c).

From the free product construction~\eqref{eq:YGfp},
we see that the identification~\eqref{eq:L2s} of Hilbert spaces identifies
$\Tc_\psi$ with the von Neumann algebra
\begin{equation}\label{eq:tailinY}
\bigcap_{N\ge1}W^*(\bigcup_{j\ge N}\pi_{\rho\circ G}(B_i)).
\end{equation}
Proposition~\ref{prop:TinW*D} implies that the von Neumann algebra~\eqref{eq:tailinY} is contained in $\pi_{\rho\circ G}(D)''$,
while the opposite inclusion is implied by the formula
\[
G(y)=\lim_{n\to\infty}\frac1n\sum_{k=N+1}^{N+n}\alpha^k(y),
\]
which follows from~\eqref{eq:Gylim}.
Thus, we have shown~(d).

Comparing the definition of $F_\psi$ (i.e., \eqref{eq:WOTlim} of Theorem~\ref{thm:E}) with~\eqref{eq:Gylim} shows
the identification of $F_\psi\restrict_{\pi_\psi(\Afr)}$ with $G$, namely~(e).

When we consider Definition~\ref{def:tailC*} of $D_\psi$ and the identifications from Table~\ref{tab:cor} that have already been verified,
and taking into account that we have, from Proposition~\ref{prop:QSSfreeprod}, $D_\psi\subseteq\pi_\psi(\Afr)$, we find that the identification
of Hilbert spaces~\eqref{eq:L2s} yields an identification of $D_\psi$ with the smallest unital $C^*$-subalgebra $Z$ of $Y$ that satisfies
\begin{equation}\label{eq:ZG}
G(z_0\theta_{i_1}(a_1)z_1\cdots\theta_{i_n}(a_n)z_n)\in Z\text{ for all }n,i_1,\ldots,i_n\in\Nats,\,z_0,\ldots,z_n\in Z,\,a_j\in A,
\end{equation}
where $\theta_i$ is $\theta$ followed by the mapping from $B$ onto the $i$-th copy of $B$ in $Y$.
We aim to show $Z=D$.
By contrast, from the condition~\ref{def:VA}\eqref{item:Dsmallest}, we have that $D$ is the smallest unital $C^*$-subalgebra
of $Y$ that satisfies
\begin{equation}\label{eq:DG}
G(x_0\theta(a_1)x_1\cdots\theta(a_n)x_n)\in D\text{ for all }n\in\Nats,\,x_0,\ldots,x_n\in D,\,a_1,\ldots,a_n\in A.
\end{equation}
Since $G$ maps $Y$ into $D$, we have $Z\subseteq D$.
However, since the conditions in~\eqref{eq:DG} are more restrictive than those in~\eqref{eq:ZG}, (requiring all $i_j$ to be the same),
we have $D\subseteq Z$.
This proves the~(f).

The identification~(g) is apparent, because $B$ is generated by $\theta(A)\cup D$.

\medskip
We have shown that $V=(B,D,F,\theta,\rho)\in\VEu(A)$ yields $\psi=\psi_V\in\QSS(A)$ making the identifications of Table~\ref{tab:cor} hold.
In the reverse direction, Proposition~\ref{prop:QSSfreeprod} shows that given $\psi\in\QSS(A)$, there is $V=V_\psi=(B,D,F,\theta,\rho)$
so that $\psi=\psi_V$ as in the construction performed above
and where $D=D_\psi$, $B=C^*(D_\psi\cup\pi_\psi\circ\lambda_1(A))$, $\theta=\pi_\psi\circ\lambda_1$, $F=F_\psi\restrict_B$ and $\rho=\psi\restrict_{D_\psi}$.

We must show that the quintuple $V$ satisfies the conditions of Definition~\ref{def:VA}.
Once we have done so, we will be able to conclude that the map~\eqref{eq:bijVQ} is onto $\QSS(A)$.
We will also be able to conclude that it is injective, because by Proposition~\ref{prop:QSSfreeprod} and its proof and by the identifications
in Table~\ref{tab:cor},
if $\psi=\psi_{V'}$ for some $V'\in\VEu(A)$, then $V'$ is recovered as $V_\psi$.

Of the conditions in Definition~\ref{def:VA} all are clearly satisfied by $V_\psi$,
except perhaps~\eqref{item:Dsmallest}.
However, from Definition~\ref{def:tailC*}, we know that $D=D_\psi$ is the smallest unital $C^*$-subalgebra $Z$ of $\pi_\psi(\Afr)$ that satisfies
\begin{equation}\label{eq:ZF}
F_\psi(z_0\pi_{i_1}(a_1)z_1\cdots\pi_{i_n}(a_n)z_n)\in Z\text{ for all }n,i_1,\ldots,i_n\in\Nats,\,z_0,\ldots,z_n\in Z,\,a_j\in A,
\end{equation}
where $\pi_i=\pi_\psi\circ\lambda_i$,
and we must show that it is the smallest unital $C^*$-subalgebra $\Dt$ of $B$ that satifies
\begin{equation}\label{eq:DtF}
F(x_0\theta(a_1)x_1\cdots\theta(a_n)x_n)\in\Dt\text{ for all }n\in\Nats,\,x_0,\ldots,x_n\in\Dt,\,a_1,\ldots,a_n\in A.
\end{equation}
Clearly, $\Dt\subseteq D_\psi$.
To show the reverse inclusion, it will suffice to show that if we assume~\eqref{eq:ZF} holds when $i_1=i_2=\cdots=i_n$, then it also holds in full generality.
However, this follows by freeness of the algebras $\stalg(\pi_i(A)\cup D_\psi)$ with respect to $F_\psi$,
either by appeal to the moment-cumulant formula of Speicher~\cite{Sp98}, or by elementary considerations and arguing by induction on $n$.
\end{proof}

\begin{cor}\label{cor:tracial}
The bijection from~\eqref{eq:bijVQ} restricts to bijections
\begin{align}
\TVEu(A)&\to\TQSS(A) \label{eq:TVA} \\
\VEu(A,\phi)&\to\QSS(A,\phi) \label{eq:VAphi} \\
\TVEu(A,\tau)&\to\TQSS(A,\tau), \label{eq:TVAtau}
\end{align}
for a state $\phi$ on $A$ and a tracial state $\tau$ on $A$,
where
\begin{align*}
\TVEu(A)&=\{(B,D,F,\theta,\rho)\in\VEu(A)\mid\rho\circ F\text{ is a faithful trace on }B\} \\
\VEu(A,\phi)&=\{(B,D,F,\theta,\rho)\in\VEu(A)\mid\rho\circ F\circ\theta=\phi\} \\
\TVEu(A,\tau)&=\TVEu(A)\cap\VEu(A,\tau).
\end{align*}
\end{cor}
\begin{proof}
To show that the map~\eqref{eq:TVA} is into $\TQSS(A)$, we use the fact that if
\[
(Y,G)=(*_D)_{i\in I}(B_i,F_i)
\]
is an amalgamated free product of von Neumann algebras and if there is a tracial state $\tau$ on $D$ 
so that for each $i\in I$, $\tau\circ F_i$ is a trace on $B_i$, then $\tau\circ G$ is a trace on $Y$.
This is a well known fact that is not difficult to verify;
it can be proved similarly to the case when $D=\Cpx$, (i.e., showing that the free product of traces is a trace);
for this latter fact see, for example,~\cite{VDN92}.

To see that the map~\eqref{eq:TVA} is onto $\TQSS(A)$, take $\psi\in\TQSS(A)$ and let $V=(B,D,F,\theta,\rho)\in\VEu(A)$
be such that $\psi=\psi_V$.
Then $\rho\circ F$ arises as the restriction of the trace $\psih$ on $\Mcal_\psi$ to a $C^*$-subalgebra identified with $B$
and is, therefore, a trace.
However, since $\psih$ has faithful GNS representation, being a trace, it must be faithful on $\Mcal_\psi$.
Thus, $\rho\circ F$ is a faithful trace, and $V\in\TVEu(A)$.

That the other maps~\eqref{eq:VAphi} and~\eqref{eq:TVAtau} are bijections, is now clear.
\end{proof}

Because traces whose GNS representations are faithful must themselves be faithful, we can use a von Neumann algebra
version of $\TVEu(A)$ to describe $TQSS(A)$.

\begin{defi}\label{def:TWA}
For a unital $C^*$-algebra $A$, let $\TWEu(A)$ be the set of all equivalence classes of quintuples
$W=(\Bc,\Dc,E,\theta,\rho)$ where
\begin{enumerate}[(i)]
\item $\Bc$ is a von Neumann algebra
\item $\Dc$ is a unital von Neumann subalgebra of $\Bc$,
\item $E:\Bc\to\Dc$ is a normal, faithful conditional expectation onto $\Dc$,
\item $\theta:A\to B$ is a unital $*$-homomorphism,
\item\label{item:ADcgen} $\theta(A)\cup\Dc$ generates $\Bc$ as a von Neumann algebra,
\item\label{item:Dcsmallest} $\Dc$ is the smallest unital von Neumann subalgebra of $\Bc$ that satisfies
\begin{equation}\label{eq:Edas}
E\big(x_0\theta(a_1)x_1\cdots\theta(a_n)x_n\big)\in\Dc
\end{equation}
whenever $n\in\Nats$, $x_0,\ldots,x_n\in\Dc$ and $a_1,\ldots,a_n\in A$,
\item $\tau$ is a normal faithful tracial state on $\Dc$, such that $\tau\circ E$ is a trace on $\Bc$.
\end{enumerate}
\end{defi}

\begin{thm}\label{thm:vNdescr}
There is a bijection
\begin{equation}\label{eq:bijTWQ}
\TWEu(A)\to\TQSS(A)
\end{equation}
that assigns to $W=(\Bc,\Dc,E,\theta,\tau)\in\TWEu(A)$ the tracial state $\psi=\psi_W$
given as follows:
letting 
\begin{equation}\label{eq:MHfp}
(\Mcal,H)=(*_\Dc)_{i=1}^\infty(\Bc,E)
\end{equation}
be the amalgamated free product of von Neumann algebras and letting $*_1^\infty\theta:\Afr\to\Mcal$ be the free product (arising from the univeral property) of the
homomorphisms from the copies of $A$ into the respective copies of $\Bc$, 
we set $\psi=\tau\circ H\circ(*_1^\infty\theta)$.

We have the following correspondence between objects defined in Section~\ref{sec:tails} and objects associated to $(\Mcal,H)$:
\begin{equation}\label{eq:vNtable}
\text{
\begin{tabular}{c|c}
from $(\Mcal,H)$ & from Section~\ref{sec:tails} \\ \hline\hline
\rule{0ex}{2.5ex}$L^2(\Mcal,\tau\circ H)$ & $L^2(\Afr,\psi)$ \\ \hline
\rule{0ex}{2.5ex}$\Mcal$ & $\Mcal_\psi$ \\ \hline
\rule{0ex}{2.5ex}$\Dc$ & $\Tc_\psi$ \\ \hline
\rule{0ex}{2.5ex}$\Bc_i$ & $W^*(\pi_\psi(\lambda_i(A))\cup\Tc_\psi)$ \\
\end{tabular}
}
\end{equation}
\medskip
\noindent
where $\Bc_i$ denotes the $i$-th copy of $\Bc$ in $\Mcal$.
\end{thm}
\begin{proof}
We use Corollary~\ref{cor:tracial}.
Since $\psih$ is a trace on $\Mcal_\psi$ whose GNS representation is faithful, $\psih$ itself must be faithful and Proposition~\ref{prop:N=M}
applies.
\end{proof}

\section{Extreme quantum symmetric states} 
\label{sec:exq}

In this section, we characterize the extreme points of the compact convex set $\QSS(A)$ in terms of the corresponding elements
of $\VEu(A)$.

\begin{thm}\label{thm:extrQSS}
Let $\psi\in\QSS(A)$ and let $V=V_\psi=(B,D,F,\theta,\rho)\in\VEu(A)$ be the corresponding quintuple
under the bijection of Theorem~\ref{thm:descr}.
Then the following are equivalent:
\begin{enumerate}[(i)]
\item\label{it:rhopure} $\rho$ is a pure state of $D$,
\item\label{it:extSS} $\psi$ is an extreme point of $\SSt(A)$,
\item\label{it:extQSS} $\psi$ is an extreme point of $\QSS(A)$,
\end{enumerate}
\end{thm}
\begin{proof}
The implication \eqref{it:rhopure}$\implies$\eqref{it:extSS} is from Proposition~\ref{prop:extSS}, while \eqref{it:extSS}$\implies$\eqref{it:extQSS}
is trivially true.
To show \eqref{it:extQSS}$\implies$\eqref{it:rhopure}, we expand on the proof of Proposition~\ref{prop:extSS}\eqref{it:extrSS=>pure}.
If $\rho$ is not a pure state of $D$, then in the decomposition $\psi=\frac12(\psi_1+\psi_2)$ constructed there,
the states $\psi_1$ and $\psi_2$ are, in fact, quantum symmetric.
They are quantum symmetric by Proposition~\ref{prop:freeqsymm}, because the family $(D_\psi\cup\pi_\psi\circ\lambda_i(A))_{i=1}^\infty$
is free with respect to $F_\psi$.
\end{proof}

\begin{thm}\label{thm:extrTQSS}
Let $\psi\in\TQSS(A)$ and let $W=(\Bc,\Dc,E,\theta,\tau)\in\TWEu(A)$ be the quintuple corresponding to $\psi$ under the bijection
of Theorem~\ref{thm:vNdescr}.
Let $R(E)$ be the set of all normal tracial states $\tau$ of $D$ such that $\tau\circ E$ is a trace of $\Bc$.
Then $\psi$ is an extreme point of $\TQSS(A)$ if and only if $\tau$ is an extreme point of $R(E)$.
\end{thm}
\begin{proof}
The proof is quite similar to that of Theorem~\ref{thm:extrQSS}.
If $\tau$ is not an extreme point of $R(E)$, then there are 
distinct elements $\tau_1$ and $\tau_2$ of $R(E)$ such
$\tau=\frac12(\tau_1+\tau_2)$.
Replacing $\tau_1$ by $(3\tau_1+\tau_2)/4$ and $\tau_2$ by $(\tau_1+3\tau_2)/4$, if necessary,
we may without loss of generality assume that both $\tau_1\circ E$ and $\tau_2\circ E$ are faithful on $\Bc$.
Thus, we have $W_i:=(\Bc,\Dc,E,\sigma,\tau_i)\in\TWEu(A)$.
For $i=1,2$, letting $\psi_i\in\TQSS(A)$ be the state corresponding to $W_i$ under the bijection
from Theorem~\ref{thm:vNdescr}, we have $\psi=\frac12(\psi_1+\psi_2)$.
Since $W_1\ne W_2$, we have $\psi_1\ne\psi_2$, and $\psi$ is not an extreme point of $\TQSS(A)$.

Suppose $\psi$ is not an extreme point of $\TQSS(A)$.
By Lemma~\ref{lem:statesum}, there are normal states $\chi_i$ on $\Mcal_\psi$ satisfying~\eqref{eq:chirestrict},
\eqref{eq:chipsi} and~\eqref{eq:psihchi}, and density of the range of $\pi_\psi$ in $\Mcal_\psi$ implies
that $\chi_1$ and $\chi_2$ are traces.
By Proposition~\ref{prop:N=M}, $F_\psi$ has a (unique) normal extension $E_\psi$ to $\Mcal_\psi$.
Let $\tau_i$ be the restriction of $\chi_i$ to $\Tc_\psi=\Dc$.
Then $\tau=\frac12(\tau_1+\tau_2)$.
Since $\psi_j=\tau_j\circ F_\psi\circ\pi_\psi$, we have $\tau_1\ne\tau_2$.
Moreover, the restriction of $\tau_i\circ E_\psi$ to $\Bc_1$ (which is the state $\tau_i\circ E$ on $\Bc$)
is a trace.
Thus, $\tau_1,\tau_2\in R(E)$
and $\tau$ is not an extreme point of $R(E)$.
\end{proof}

\begin{remark}
In the above theorem, if either $\Bc$ or $\Dc$ is a factor, then $\psi$ is an extreme point of $\TQSS(A)$.
\end{remark}

Here is an example of an extreme tracial quantum symmetric state for $A=\Cpx\oplus\Cpx$ where the tail algebra
is noncommutative.
\begin{example}
We describe an element $(\Bc,\Dc,E,\theta,\tau)$ of $\TWEu(A)$ where $A$ is the two-dimensional $C^*$-algebra.
Let $\Bc=M_2(\Cpx)\oplus M_2(\Cpx)$, identified with $(\Cpx\oplus\Cpx)\otimes M_2(\Cpx)$, and let
$\Dc\subseteq\Bc$ be the copy of $M_2(\Cpx)$ identified with $1\otimes M_2(\Cpx)$.
Let $E:\Bc\to\Dc$ be the conditional expectation $E=\phi\otimes\id_{M_2(\Cpx)}$, where $\phi$ is the state on $\Cpx\oplus\Cpx$
given by $\phi(a\oplus b)=(2a+b)/3$.
Let $\theta:A\to\Bc$ be the unital $*$-homomorphism determined by $\theta(1\oplus0)=p$, where
\[
p=\begin{pmatrix}1&0\\0&0\end{pmatrix}\oplus
 \frac12\begin{pmatrix}1&1\\1&1\end{pmatrix}.
\]
Then
\[
E(p)=1\otimes\frac16\begin{pmatrix}5&1\\1&1\end{pmatrix}
\]
while
\[
p\,E(p)\,p=\frac16\begin{pmatrix}5&0\\0&0\end{pmatrix}\oplus\frac13\begin{pmatrix}1&1\\1&1\end{pmatrix},
\qquad
E(p\,E(p)\,p)=1\otimes\frac19\begin{pmatrix}6&1\\1&1\end{pmatrix}.
\]
Clearly, $E(p)$ and $E(p\,E(p)\,p)$ together generate $\Dc$ as an algebra, while $\Dc$ and $p$ together generate $\Bc$.
Thus, conditions~\eqref{item:ADcgen} and~\eqref{item:Dcsmallest} of Definition~\ref{def:TWA} are satisfied
and $(\Bc,\Dc,E,\theta,\tr_2)\in\VEu(A)$, where $\tr_2$ is the normalized trace on $\Dc$.
Since $\Dc\cong M_2(\Cpx)$ is a factor, this quintuple yields an extreme point of $\TQSS(A)$, by Theorem~\ref{thm:extrTQSS}.
\end{example}

Here is an example of an extreme tracial quantum symmetric state where neither $\Bc$ nor $\Dc$ in the corresponding
element of $\TWEu(A)$ is a factor.
\begin{example}
Let $A=\Cpx\oplus\Cpx$ with $q=1\oplus0\in A$, let $\Bc=M_2(\Cpx)\oplus M_2(\Cpx)$, let $\Dc\cong\Cpx\oplus\Cpx$
be the subalgebra
\begin{equation}\label{eq:summats}
\Dc=\left\{\begin{pmatrix}\lambda&0\\0&\mu\end{pmatrix}\oplus\begin{pmatrix}\lambda&0\\0&\mu\end{pmatrix}
\biggm|\lambda,\mu\in\Cpx\right\}
\end{equation}
of $\Bc$ and let $E:\Bc\to\Dc$ be the trace-preserving conditional expectation that sends
\[
\begin{pmatrix}a_{11}&a_{12}\\a_{21}&a_{22}\end{pmatrix}\oplus\begin{pmatrix}b_{11}&b_{12}\\b_{21}&b_{22}\end{pmatrix}
\]
to the element as indicated in~\eqref{eq:summats} with $\lambda=\frac12(a_{11}+b_{11})$ and $\mu=\frac12(a_{22}+b_{22})$.
Clearly, the unique state $\tau$ on $\Dc$ making $\tau\circ E$ a trace on $\Bc$ is the one assigning equal weights
of $1/2$ to the minimal projections of $\Dc$, and then $\tau\circ E$ is a faithful trace.
Let $\theta:A\to\Bc$ be the $*$-homomorphism that sends $q$ to
\[
\theta(q)=
\frac13\begin{pmatrix}1&\sqrt2\\\sqrt2&2\end{pmatrix}
\oplus
\frac14\begin{pmatrix}1&\sqrt3\\\sqrt3&3\end{pmatrix}.
\]
Then $E(\theta(A))=\Dc$, so condition~\eqref{item:Dcsmallest} of Definition~\ref{def:TWA} is fulfilled.
Now we easily see that $\Dc$ and $\theta(q)$ together generate $\Bc$, so condition~\eqref{item:ADcgen}
of Definition~\ref{def:TWA} is fulfilled.
Thus, the quintuple
$W=(\Bc,\Dc,E,\theta,\tau)$ belongs to $\TWEu(A)$
and, by Theorem~\ref{thm:extrTQSS}, the corresponding $\psi\in\TQSS(A)$ is an extreme point of $\TQSS(A)$.
\end{example}

St\o{}rmer~\cite{St69}
found that all of the extreme symmetric states on the minimal tensor product $C^*$-algebra $\bigotimes_1^\infty A$
are tensor product states $\otimes_1^\infty\phi$ for $\phi$ in the state space $S(A)$ of $A$.
The analogy of this in our setting
is the natural embedding of $S(A)$ into the set of extreme points of $\QSS(A)$, using the free product construction
(with amalgamation over the scalars) but the image of this embedding is, of course, far from being the set of
all the extreme points of $\QSS(A)$.
This situation is summarized in the following proposition.

Traditionally, when we take the reduced free product
\begin{equation}\label{eq:Afp}
(\Ac,\varphi)=\freeprod_{i=1}^\infty(A,\phi)
\end{equation}
of $C^*$-algebras, we insist that the GNS representation of $\phi$  be faithful,
which ensures that we have embeddings of $A$ into $\Ac$ (one for each $i\in\Nats$).
However, we will continue to use the notation~\eqref{eq:Afp} when the GNS representation of $\phi$
is not faithful, to mean the reduced free product
\[
(\Ac,\varphi)=\freeprod_{i=1}^\infty(\At,\phit)
\]
where $\At$ is the quotient of $A$ by the kernel $I_\phi$ of the GNS representation of $\phi$, and $\phit$ is the state on $\At$
whose composition with the quotient map yields the state $\phi$ on $A$.
Now instead of embeddings of $A$ into $\Ac$, we have $*$-homomorphisms $q_i:A\to\Ac$,
one for each $i\in\Nats$, from $A$ into $\Ac$, whose kernels are $I_\phi$.
For $\phi\in S(A)$, let $\freeprod_1^\infty\phi$ be the reduced free product state on $\Afr$, obtained by
taking the reduced free product $(\Ac,\varphi)=\freeprod_{i=1}^\infty(A,\phi)$ of $C^*$-algebras
and composing 
the canonical quotient map $\Afr\to\Ac$
(sending the $i$-th copy of $A$ in $\Afr$ into $\Ac$ by the $*$-homomorphism $q_i$),
with the free product state $\varphi$ on $\Ac$.

\begin{prop}\label{prop:fp}
The map $\phi\mapsto\freeprod_1^\infty\phi$ sends $S(A)$ onto the set of all states in $\QSS(A)$
whose tail $C^*$-algebras are copies of $\Cpx$.
The latter are extreme points of $\QSS(A)$.
Furthermore, the restriction of this map to the set $TS(A)$ of all tracial states of $A$, maps $TS(A)$
onto the set of all traces in $\TQSS(A)$ whose tail algebras are copies of $\Cpx$, and the latter are extreme
points of $\TQSS(A)$.
\end{prop}
\begin{proof}
The first two sentences follow easily from the details of the proof of
Theorem~\ref{thm:descr}.
The last sentence follows from the first assertion and the well known fact that the free product of traces is a trace.
\end{proof}

Another embedding of $S(A)$ into $\QSS(A)$ is the affine embedding that corresponds to taking ``maximal amalgamation,'' as follows.
Let $\qchk:\Afr\to A$ be the quotient map resulting from the univeral property
that sends each copy of $A$ in $\Afr$ identically to $A$.
Given $\phi\in S(A)$, 
let $\phichk$ be the state on $\Afr$ that is $\phichk=\phi\circ\qchk$.

\begin{prop}\label{prop:amalgaA}
The map $\phi\mapsto\phichk$ is an affine embedding of $S(A)$ into $\QSS(A)$
and, $\phichk$ is an extreme element of $\QSS(A)$ if and only if $\phi$ is a pure state of $A$.
The restriction of this map to the tracial state space $TS(A)$ yields an affine embedding of $TS(A)$ into $\TQSS(A)$
and for $\tau\in TS(A)$, $\tauchk$ is an extreme point of $\TQSS(A)$ if and only if $\tau$ is an extreme point of $TS(A)$.
\end{prop}
\begin{proof}
That $\phi\mapsto\phichk$ is an affine embedding is clear.
Thus, if $\phi$ is not a pure state of $A$, then $\phichk$ is not an extreme point in $\QSS(A)$.
The element $V=(B,D,F,\theta,\rho)\in\VEu(A)$ corresponding to the quantum symmetric state $\phichk$ has
\renewcommand{\labelitemi}{$\bullet$}
\begin{itemize}
\item $B=A$
\item $D=B$,
\item $F$ the identity map,
\item $\theta$ the identity map
\item $\rho=\phi$.
\end{itemize}
Thus, $\rho$ is extreme if and only if $\phi$ is pure.
\end{proof}

The tracial state space of a unital $C^*$-algebra, if nonempty, forms a Choquet simplex
(see, for example, Theorem 3.1.18 of~\cite{Sa71}).
So the following question seems natural.
\begin{ques}
Is $\TQSS(A)$ a Choquet simplex for every unital $C^*$-algebra $A$ that has a tracial state?
\end{ques}

We also have the following results, whose proofs are similar to those of Theorems~\ref{thm:extrQSS} and~\ref{thm:extrTQSS}.

\begin{thm}\label{thm:extrQSSphi}
Let $\phi$ be a state on $A$, let
$\psi\in\QSS(A,\phi)$ and let $v=(B,D,F,\theta,\rho)\in\VEu(A,\phi)$ be the quintuple corresponding to $\psi$
under the bijection of Corollary~\ref{cor:tracial}.
Let $S(F,\phi)$ be the set of all states $\eta$ on $D$ so that $\eta\circ F\circ\theta=\phi$.
Then $\psi$ is an extreme point of $\QSS(A,\phi)$ if and only if $\rho$ is an extreme point of $S(F,\phi)$.
\end{thm}

\begin{thm}\label{thm:extrTQSStau}
Let $\tau$ be a tracial state on $A$, let
$\psi\in\TQSS(A,\tau)$ and let $W=(\Bc,\Dc,E,\theta,\rho)\in\TWEu(A,\tau)$ be the quintuple corresponding to $\psi$
under the bijection of Theorem~\ref{thm:extrTQSS}.
Let $R(E,\tau)$ be the set of all normal, tracial states $\eta$ on $\Dc$ so that $\eta\circ E$ is a trace on $\Bc$ and
$\eta\circ E\circ\theta=\tau$.
Then $\psi$ is an extreme point of $\TQSS(A,\tau)$ if and only if $\rho$ is an extreme point of $R(E,\tau)$.
\end{thm}

\section{Central quantum symmetric states}
\label{sec:ZQSS}

\begin{defi}
Let $\psi\in\QSS(A)$ be a quantum symmetric state on $\Afr=*_1^\infty A$.
With reference to Definition~\ref{def:tail},
we say that $\psi$ is {\em central} if the tail algebra of $\psi$ lies in the center of  $\Mcal_\psi$.
We let $\ZQSS(A)$ denote the set of all central quantum symmetric states on $\Afr$.
The tracial central quantum symmetric states are those in the set $\ZTQSS(A):=\TQSS(A)\cap\ZQSS(A)$.
\end{defi}

With reference to Proposition~\ref{prop:fp}, we see that for every $\phi\in S(A)$, the free product state $*_1^\infty\phi$
is a central quantum symmetric state.
The next result is related to Proposition~4.7 of~\cite{DK}, though it is not precisely a generalization of it
(since the states considered in Proposition~4.7 of~\cite{DK}, including the free product states, are faithful).

\begin{thm}\label{thm:ZQSS}
The sets $\ZQSS(A)$ and $\ZTQSS(A)$
are compact, convex subsets of $\QSS(A)$ and and both are Choquet simplices.
Their extreme points are the free product states and free product tracial states, respectively:
\begin{align}
\partial_e(\ZQSS(A))&=\{*_1^\infty\phi\mid \phi\in S(A)\}, \label{eq:extrZQSSA} \\
\partial_e(\ZTQSS(A))&=\{*_1^\infty\tau\mid \tau\in TS(A)\}. \label{eq:extrZTQSSA} 
\end{align}
\end{thm}
\begin{proof}
We will focus first on $\ZQSS(A)$ and we will show
\begin{equation}\label{eq:ZQ=CC}
\ZQSS(A)=\overline{\conv}\{*_1^\infty\phi\mid \phi\in S(A)\}.
\end{equation}
To show the inclusion $\subseteq$,
let $\psi\in\ZQSS(A)$
By Theorem~\ref{thm:descr}, $\psi=\rho\circ G\circ(*_1^\infty\theta)$
where
\[
(Y,G)=(*_D)_1^\infty(B,F)
\]
and $\theta:A\to B$ a unital $*$-homomorphism.
By hypothesis, the tail algebra $D$ lies in the center of $Y$;
in particular $D$ is commutative, isomorphic to $C(X)$ for some compact Hausdorff space $X$, and $\rho$ is in the weak$^*$-closed convex
hull of the set of point evaluation maps $\ev_x$, ($x\in X$).
By Lemma~4.2 of~\cite{DK}, each of the compositions $\ev_x\circ G\circ(*_1^\infty\theta)$ is a free product state $*_1^\infty\phi_x$ on $\Afr$,
where $\phi_x=\ev_x\circ F\circ\theta$.
Thus, taking a net of finite convex combinations of the $\ev_x$ that converges in weak$^*$-topology of $C(X)^*$
to $\rho$, we get that the net
of the corresponding convex combinations of the free product states $*_1^\infty\phi_x$ converges in the weak$^*$-topology on
$\Afr^*$
to $\psi$.
This proves $\subseteq$ in~\eqref{eq:ZQ=CC}.

We will now show the inclusion $\supseteq$ in~\eqref{eq:ZQ=CC}.
The mapping 
\begin{equation}\label{eq:fpmap}
S(A)\ni\phi\mapsto*_1^\infty\phi\in S(\Afr)
\end{equation}
is continuous with respect to the respective restrictions of weak$^*$-topologies.
This follows from the facts that the $*$-subalgebra, $\Afr_0$, of $\Afr$
that is generated by $\bigcup_{i=1}^\infty\lambda_i(A)$ is norm dense in $\Afr$, and, by freeness,
for each element $x\in\Afr_0$ there is a list $a_1,\ldots,a_k$ of elements of $A$ and a polynomial $p$ such that
$(*_1^\infty\phi)(x)=p(\phi(a_1),\ldots,\phi(a_k))$.
Thus, $\{*_1^\infty\phi\mid\phi\in S(A)\}$ is a compact subset of $S(\Afr)$ and the mapping~\eqref{eq:fpmap},
since it is clearly injective, is a homeomorphism onto its image.

By a classical result (see, for example, Proposition 1.2 of \cite{Ph01}) using the compactness of the set
\begin{equation}\label{eq:fpst}
\{*_1^\infty\phi\mid \phi\in S(A)\},
\end{equation}
an arbitrary element $\psi$ of the right-hand-side of~\eqref{eq:ZQ=CC} is
the barycenter of a Borel probability measure on the set~\eqref{eq:fpst}, which under 
the mapping~\eqref{eq:fpmap} corresponds to a Borel probability measure $\mu$ on $S(A)$.
Thus, for all $x\in\Afr$ we have
\begin{equation}\label{eq:psimu}
\psi(x)=\int_{S(A)}(*_1^\infty\phi)(x)\,d\mu(\phi).
\end{equation}
We will use this to show that the tail algebra of $\psi$ lies in the center of $\Mcal_\psi$.

Let $C(S(A))\otimes A$ denote the (minimal) $C^*$-algebra tensor product.
Identify
$C(S(A))$ with the unital $C^*$-subalgebra $C(S(A))\otimes 1$ of $C(S(A))\otimes A$.
Let
$E:C(S(A))\otimes A\to C(S(A))$ be the conditional expectation given by
\[
E(f\otimes a)(\phi)=f(\phi)\phi(a),\qquad (f\in C(S(A)),\,a\in A,\,\phi\in S(A)).
\]
Let $B$ be the quotient of $C(S(A))\otimes A$ under the GNS representation
associated to $E$, namely, on the Hilbert $C(S(A))$-module $L^2(C(S(A))\otimes A,E)$.
Then $C(S(A))$ is contained in the center of $B$.
We have
a conditional expectation $F:B\to C(S(A))$ arising from $E$, that we can view as compression with respect to the projection
$L^2(C(S(A))\otimes A,E)\to C(S(A))$.
In fact, (though we will not use this description)
$B$ is a continuous field of $C^*$-algebras over $S(A)$, and the fiber over $\phi\in S(A)$ is isomorphic
to the image of $A$ under the GNS representation of $\phi$.
Let
\begin{equation}\label{eq:BF}
(Y,G)=(*_{C(S(A))})_{i=1}^\infty(B,F)
\end{equation}
be the reduced amalgamated free product of $C^*$-algebras, with amalgamation over $C(S(A))$.
Then $C(S(A))$ lies in the center of $Y$.

Let $\theta:A\to B$ be the unital $*$-homomorphism that is the composition of the mapping $A\to C(S(A))\otimes A$
given by $a\mapsto 1\otimes a$ and the quotient mapping of $C(S(A))\otimes A$ to $B$.

We consider the $*$-homomorphism $*_1^\infty\theta:\Afr\to Y$  sending $\lambda_i(A)$ to the $i$-th copy of $B$ in $Y$ via $\theta$.
Then $\psi=\rho\circ G\circ(*_1^\infty\theta)$, where $\rho$ is the state on $C(S(A))$ given by integration with respect to $\mu$.

Note that $\theta(A)\cup C(S(A))$ generates $B$ as a $C^*$-algebra.
Moreover, 
since the family $E(1\otimes A)$ separates points of $C(S(A))$, by the Stone-Weierstrass Theorem, $E(1\otimes A)$ generates $C(S(A))$,
so $F(B)$ generates $C(S(A))$.
Thus, $(*_1^\infty\theta)(\Afr)=Y$.
Consequently, we may identify $L^2(\Afr,\psi)$ and $L^2(Y,\rho\circ G)$ and thereby identify the images $\pi_\psi(\Afr)$ and $\pi_{\rho\circ G}(Y)$
of the GNS representations.
Furthermore, the tail algebra $\Tc_\psi$ of $\psi$ is identified with
\[
\Tc_{\rho\circ G}:=\bigcap_{N\ge1}W^*(\bigcup_{j\ge N}\pi_{\rho\circ G}(B_i)).
\]
By Proposition~\ref{prop:TinW*D}, $\Tc_{\rho\circ G}$ lies the von Neumann algebra generated by $\pi_{\rho\circ G}(C(S(A)))$, which is in the center of 
$\pi_{\rho\circ G}(Y)''$.
Consequently, $\psi$ is a central quantum symmetric state.
This completes the proof of~\eqref{eq:ZQ=CC};
in particular, $\ZQSS(A)$ is a closed, convex subset of $\QSS(A)$ and is, thus, compact.

{}From Proposition~\ref{prop:fp}, we have that every free product state $*_1^\infty\phi$ is an extreme point of $\QSS(A)$
and is, therefore, an extreme point of $\ZQSS(A)$.
{}From~\eqref{eq:ZQ=CC}, we deduce that there are no other extreme points of $\ZQSS(A)$, and~\eqref{eq:extrZQSSA} is proved.

To show that $\ZQSS(A)$ is a Choquet simplex,
we suppose $\mu$ and $\nu$ are Borel probability measures on $S(A)$ and
\[
\int_{S(A)}(*_1^\infty\phi)(x)\,d\mu(\phi)=
\int_{S(A)}(*_1^\infty\phi)(x)\,d\nu(\phi)
\]
for all $x\in\Afr$, and we will show $\mu$ must equal $\nu$.
Taking $x=\lambda_1(a_1)\lambda_2(a_2)\cdots\lambda_n(a_n)$ for $a_1,\ldots,a_n\in A$, we have
$(*_1^\infty\phi)(x)=\prod_{j=1}^n\phi(a_j)$.
Thus, $\int f\,d\mu=\int f\,d\nu$ for all functions $f$ in the subalgebra of $C(S(A))$ generated by the set
of all functions $S(A)\ni\phi\mapsto\phi(a)$ as $a$ ranges over $A$.
This subalgebra contains the constants and separates points of $S(A)$, so, by the Stone--Weierstrass theorem, it is dense
with respect to the uniform norm in $C(S(A))$.
We deduce $\mu=\nu$, so $\ZQSS(A)$ is indeed a Choquet simplex.

\smallskip
Now we focus on $\ZTQSS(A)$.
{}From $\ZTQSS(A)=\ZQSS(A)\cap\TQSS(A)$, we see that $\ZTQSS(A)$
is a closed, convex subset of $\ZQSS(A)$, and the free product traces
are extreme points of $\ZTQSS(A)$.
{}From what we have already shown about $\ZQSS(A)$, in order to show that $\ZTQSS(A)$ is a Choquet simplex whose
extreme points are the free product traces as in~\eqref{eq:extrZTQSSA},
it will suffice to show that if $\mu$ is a Borel probability meaure on $S(A)$
and if $\psi\in\ZQSS(A)$ given by~\eqref{eq:psimu} is a trace, then $\mu$ is supported in $TS(A)$.
Let $a\in A$ have $\|a\|\le1$ and let $\omega$ be the push-forward measure of $\mu$ under the map
\[
S(A)\to[0,1]^2,\qquad\phi\mapsto\big(\phi(a^*a),\phi(aa^*)\big).
\]
It will suffice to show that $\omega$ is supported in the diagonal of $[0,1]^2$.
Recall the standard notation $|a|=(a^*a)^{1/2}$ and $|a^*|=(aa^*)^{1/2}$.
Letting
\[
x=\lambda_1(|a|)\lambda_2(a),\qquad y=\lambda_1(|a^*|)\lambda_2(a^*),
\]
we get
\begin{alignat*}{2}
(*_1^\infty\phi)(x^*x)&=\phi(a^*a)^2,\qquad&(*_1^\infty\phi)(xx^*)&=\phi(a^*a)\phi(aa^*), \\
(*_1^\infty\phi)(y^*y)&=\phi(aa^*)^2,\qquad&(*_1^\infty\phi)(yy^*)&=\phi(a^*a)\phi(aa^*).
\end{alignat*}
Thus, we have
\begin{align*}
\int_{[0,1]^2}s^2\,d\omega(s,t)=\psi(x^*x)
&=\psi(xx^*)=\int_{[0,1]^2}st\,d\omega(s,t), \\
\int_{[0,1]^2}t^2\,d\omega(s,t)=\psi(y^*y)
&=\psi(yy^*)=\int_{[0,1]^2}st\,d\omega(s,t).
\end{align*}
{}From these identities, we get
$\int(s-t)^2\,d\omega(s,t)=0$ and we conclude that the support of $\omega$ lies in the diagonal of $[0,1]^2$.
\end{proof}

\end{document}